\documentclass[11pt]{amsart}
\usepackage{amsmath}
\usepackage{amssymb}
\usepackage{amsthm}
\usepackage{mathrsfs}
\usepackage{comment}
\usepackage{hyperref}

\usepackage[all,cmtip]{xy}\usepackage{xcolor}
\usepackage{enumerate}
\usepackage{bm}
\usepackage{dsfont}
\usepackage{mathtools}
\usepackage[cal=euler]{mathalfa}
\usepackage[top=1in, bottom=1.25in, left=1.25in, right=1.25in]{geometry}
\usepackage{parskip}

\hypersetup{colorlinks=true,linkcolor=magenta,citecolor=blue}

\DeclareMathAlphabet{\mathpzc}{OT1}{pzc}{m}{it}

\theoremstyle{plain}
\newtheorem{theorem}{Theorem}[section]
\newtheorem*{theorem*}{Theorem}

\newtheorem{lemma}[theorem]{Lemma}
\newtheorem*{claim*}{Claim}
\newtheorem{proposition}[theorem]{Proposition}

\newtheorem{corollary}[theorem]{Corollary}
\newtheorem{conjecture}[theorem]{Conjecture}

\theoremstyle{definition}

\newtheorem{example}[theorem]{Example}

\newtheorem{remark}[theorem]{Remark}

\numberwithin{equation}{section}
\numberwithin{figure}{section}

\DeclareMathOperator{\Fix}{Fix}

\newcommand{\normKFim}{N_{K/F}(K^\times)}

\newcommand{\Cent}{\mathrm{Cent}}

\newcommand{\diag}{\textrm{diag}}

\newcommand{\GL}{\mathrm{GL}}
\newcommand{\Gal}{\mathrm{Gal}}

\newcommand{\Z}{\mathbb{Z}} 

\newcommand{\Aut}{\mathrm{Aut}}

\newcommand{\os}{n}  
\newcommand{\de}{m}  

\newcommand{\Kone}[1]{K[t;\sigma]/K[t;\sigma](t^{\de}-#1)}

\newcommand{\ignore}[1]{}

\begin{document}
\setlength{\parindent}{0pt}
\title{A closer look at some cyclic semifields}

\author{Susanne Pumpl\"un}
\address{School of Mathematical Sciences, University of Nottingham,
Nottingham NG7 2RD, United Kingdom}
\email{Susanne.Pumpluen@nottingham.ac.uk}
\thanks{Corresponding author: Susanne Pumpl\"un. Email: \href{mailto:Susanne.Pumpluen@nottingham.ac.uk}{Susanne.Pumpluen@nottingham.ac.uk}}



\date{\today}

\begin{abstract}
 We prove that the choice of the automorphism $\sigma$ in the construction of a Petit semifield is important when counting isomorphism classes.
Different choices of generators $\sigma$ of the Galois group of  a cyclic field extension $K/F$ of degree $n$ can produce non-isomorphic proper nonassociative Petit algebras $K[t;\sigma]/K[t;\sigma](t^m-a)$ when $n\nmid m$, or when $n=m$ and $a$ does not lie in the fixed field of $\sigma$. In particular, this is the case when $F$ is a finite field.
Thus there are  $\varphi(n)$ families of these semifields (where $\varphi$ is the Euler function);  one  family for each distinct generator $\sigma$ involved in their construction, and no algebra in one family is isomorphic to one in another family.

We then turn to the isotopy classes of these algebras.
When the order of $\sigma$ is greater or equal to the degree $m$, and whenever $m\geq 3$ when $n>m$,  two proper nonassociative Petit division algebras $K[t;\sigma]/K[t;\sigma](t^m-a_1)$ and $K[t;\sigma]/K[t;\sigma](t^m-a_2)$ which are isotopic are indeed already isomorphic, whenever ${\Aut}(K)$ is abelian.
 Under the same assumptions on $n$ and $m$, two such semifields constructed from the same generator $\sigma$ of ${\rm Gal}(\mathbb{F}_{q^n}/\mathbb{F}_{q})$  are thus isomorphic whenever they are isotopic.  For $n = m$, we show that each isotopy class of Petit semifields $K[t;\sigma]/K[t;\sigma]f$ contains a proper nonassociative cyclic algebra $(K/F,\sigma,a)$. We then count the isotopy classes of all Petit semifields when $n>2$.
  
 We also present some upper bounds, counts, and parametrizations for different equivalence and isomorphism classes.
Most of our results are proved in all generality for any cyclic
Galois field extension.
\end{abstract}

\maketitle

\section{Introduction}

After Sandler defined a class of semifields in 1962  \cite{Sandler1962}, a vast generalisation of his construction was presented by Petit \cite{P66, P68}.

  Petit algebras are unital nonassociative algebras that are constructed using a skew polynomial ring $R=D[t;\sigma,\delta]$ and  some monic polynomial $f\in D[t;\sigma,\delta]$,  where $D$ is a suitable unital associative ring, employing the right division by $f$ in $R$ to define the algebra multiplication. They can be viewed as canonical nonassociative generalizations of associative quotient algebras. In this paper, we will assume that $D=K$ is a field and $\sigma\in \Aut(K)$ an automorphism of degree $n$ such that $K/F$  is a cyclic Galois field extension, where $F={\rm Fix}(\sigma)$.

  For finite base fields,
    Petit algebras yield a large class of  semifields  \cite{LS} which are isotopic to the cyclic semifields  (also called Jha-Johnson semifields) introduced in \cite{JJ}, and have associated Maximum Rank Distance (MRD) codes, e.g. see  \cite{Lobillotal2026}. In this setting, their classification is mostly interesting up to isotopies.

A  classification of cyclic semifields (and hence of proper nonassociative Petit division algebras over finite fields) up to isotopy, together with the full determination of their autotopism groups has recently been obtained by Santonastaso and Zhou \cite{SantonastasoZhou2026} for fixed $\sigma$, completing initial work started by Dempwolff \cite{Demp2011, Demp2010}.
 Their paper includes a classification of the associated finite MRD codes up to linear and
semilinear equivalence over their respective prime fields.

In this paper we combine several approaches, and study families of Petit algebras as nonassociative rings both up to isotopy and isomorphism, also employing distinct generators of the Galois group in the construction.

The question whether two Petit rings $K[t;\sigma_1]/K[t;\sigma_1]f_1$ and $K[t;\sigma_2]/K[t;\sigma_2]f_2$ are isotopic that are constructed employing distinct generators $\sigma_1$ and $\sigma_2$ of ${\rm Gal}(K/F)$, was raised  for the first time in \cite[Remark 6]{LS}, in the context of semifields.

 Our contribution to this question is  Theorem \ref{thm:mainmain}, where we treat the case that $f(t)=t^m-a$ and show when distinct  generators $\sigma_1$  and $\sigma_2$  of ${\rm Gal}(K/F)$  produce non-isomorphic Petit rings. These do not need to be division rings.
  Example \ref{ex:rel} shows they can be isotopic.

 \begin{theorem*} (Theorem \ref{thm:mainmain})
Let $n\nmid m$, or let $n=m$ and $a_2\in K\setminus  F$. Every ring isomorphism
$$\varphi: K[t_1;\sigma_1]/K[t_1;\sigma_1](t_1^\de-a_1) \to K[t_2;\sigma_2]/K[t_2;\sigma_2](t_2^\de-a_2)$$
such that $\varphi$ is $F_0$-linear where $F_0\subset F$ is a subfield of $K$, satisfies
$$\varphi|_K\circ\sigma_1=\sigma_2\circ\varphi|_K.$$ 
In particular
$\sigma_1=\sigma_2$ whenever $\varphi|_K$ commutes with $\sigma_2$, e.g. when
$\Aut_{F_0}(K)$ is abelian, or when $F_0=F$.
\end{theorem*}

 We know there are $\varphi(n)$ different generators of the Galois group of $K/F$ where $\varphi$ is the Euler function.
 Thus there are  $\varphi(n)$ different families of isomorphism classes of proper Petit semifields $K[t;\sigma_i]/K[t;\sigma_i](t^m-a)$; one for each choice of generator $\sigma_i$ of the cyclic Galois group $\Gal(K/F)$, whenever either $n \nmid m$ or, alternatively, when $n =m $ and  $a\in K\setminus  F$.
Semifields in two different families cannot be isomorphic (Corollary \ref{cor:Eulerfinite}).

 We know that every Petit ring $K[t;\sigma]/K[t;\sigma]f$ is isotopic to  $K[t;\sigma^{-1}]/ \bar f K[t;\sigma^{-1}] $, where $\bar f$ is the reciprocal of $f$. This was shown already by Kantor and Liebler \cite{KL}. Indeed, the anti-isomorphism $K[t;\sigma]\to K[t;\sigma^{-1}]^{op}$,
$\sum k_it^i\mapsto\sum t^ik_i$, shows that
$(K[t;\sigma]/K[t;\sigma](t^\de-a))^{op}\cong K[t;\sigma^{-1}]/(t^m-a')K[t;\sigma^{-1}]$, which, however, for proper nonassociative Petit rings is a \emph{right} quotient ring defined using left division by $f$ and its remainders, and is not a proper nonassociative Petit ring.  This is consistent with our results.

 Theorems \ref{thm:important} and \ref{thm:mnotn} show that when $n=m$ or $n>m>2$, two proper nonassociative division algebras $K[t;\sigma_1]/K[t;\sigma_1](t^m-a_1)$ and $K[t;\sigma_2]/K[t;\sigma_2](t^m-a_2)$ constructed using the same generator $\sigma=\sigma_1=\sigma_2$
are isotopic if and only if they are isomorphic. This changes for distinct generators: although
 there are no isomorphisms between algebras with
$\sigma_1\neq\sigma_2$ by Theorems \ref{thm:firstmain} and \ref{thm:main1}, isotopies may exist (Example \ref{ex:rel}). This also seems to only hold when $f\in K[t;\sigma]$ has the special form  $f(t)=t^m-a$ (Example \ref{ex:rel2}).

 We then produce some upper bounds and counts for when the base field is finite in Section \ref{sec:counting}, which heavily depend on the order $n$ of $\sigma$ and the degree $m$ of $t^m-a$, and which can be seen as complementary to the upper bounds in \cite{LS} which covered all Petit semifields, and have a completely different flavour to the techniques employed in \cite{Lobillotal2026}. Note that semifields of the type $\mathbb{F}_{q^n}[t;\sigma]/\mathbb{F}_{q^n}[t;\sigma](t^m-a) $ are the opposite algebras (i.e., the \emph{dual semifields}) of \emph{Sandler's semifields} \cite{Sandler1962}. 
 
 When $\os\mid \de$ we can count  the number of isometry classes of  Petit rings $\Kone{a}$ that are \emph{proper nonassociative}  as $a$ runs over $K^\times$, where an isometry class contains all the algebras that are isomorphic to each other via some monomial isomorphism of degree $k$ that maps $t$ to $\alpha t^k$ for some $\alpha\in K^\times$ and restricts on $K$ to some $\tau\in \Aut (K)$ (Theorem \ref{thm:nonassociativecase}).  Our main result here is Theorem \ref{thm:cyclicmain}.

 For $m=n$, the algebras $K[t; \sigma]/K[t; \sigma](t^m - a)$ are also denoted by  $(K/F,\sigma,a)$  and called
\emph{nonassociative cyclic  algebras of degree $m$};  for $a\in F^\times$, the algebra $(K/F,\sigma,a)$ is a classical cyclic algebra
of degree $m$ as defined in \cite[p. 414]{KMRT}.

 \begin{theorem*}(Theorem \ref{thm:cyclicmain})
Let $n=m$ and $s = \gcd(q^\os-1, \frac{q^\de-1}{q-1})$. Let $e$ be a proper divisor of $n$ and put $s_e=\frac{q^e-1}{q-1}$. Then the number of isomorphism (and isotopy) classes of semifields
$(\mathbb{F}_{q^n}/\mathbb{F}_q,\sigma,a)$ for a fixed choice of $\sigma$ is exactly
\[
 M_0 =\sum_{\substack{d\mid s\\ d\nmid s_e\ \forall e\mid n,\ e<n}}
  \frac{\varphi(d)}{o(p)_d}.
\]
Varying the choice of $\sigma$ yields a total of $\varphi(n) M_0$  isomorphism classes, in $\varphi(n)$ families.
\end{theorem*}

 For $n=m\geq 2$, this result means that we can now obtain an elegant formula for  the exact number of isotopy classes of general Petit semifields $K[t;\sigma]/K[t;\sigma]f$, which Santonastaso and Zhou describe as an orbit count (Theorem \ref{thm:allcyclicmain}), and then generalize this approach to explicitly count all Petit semifields for any $n>2$, $m> 1$ (Theorem \ref{thm:orbitcounting}). Since the formula we obtain can be hard to compute, we give some bounds for the exact number of isotopy classes in Proposition \ref{prop:bounding}.

The paper is organized as follows. After the preliminaries in Section \ref{sec:prel}, Section \ref{sec:main} contains the work towards our First  Main Theorem, cited above.  In Section \ref{sec:isotopy}, we show when isotopic Petit algebras are isomorphic. In Section \ref{sec:counting}, we count the number of semifields in different ways and then use the main results to count all Petit semifields in Section \ref{sec:new}. In Section \ref{sec:parametrization} we
 give a  parametrization of the algebras $(\mathbb{F}_{q^n}/\mathbb{F}_{q},\sigma_1,a_1)$
(Theorems \ref{Main:general} and \ref{T:p2}) under the assumption that for $m>2$, assuming that $F$ contains a primitive $m$th root of unity.
 By Theorem \ref{thm:allcyclicmain}, for $n=m\geq 2$ every parametrization of the proper nonassociative cyclic $F$-algebras $(K/F,\sigma,a)$
 is also a parametrization of the $F$-isotopy classes  of all the semifields
$K[t;\sigma]/K[t;\sigma]f$, where $f$ is irreducible of degree $n$, for fixed $\sigma$; each such class contains a proper nonassociative cyclic algebra $(K/F,\sigma,a)$. 

We finish with an outlook on what our results imply for skew polycyclic codes in Section \ref{sec:codes}.

 We note that Petit rings with non-trivial left principal ideals recently gained importance in coding theory and in cryptography. They are successfully employed in the classification of skew polycyclic codes where they appear as ambient rings of the linear error-correcting codes \cite{Pum2025}. Especially the rings $K[t;\sigma]/K[t;\sigma](t^m-a)$ are the ambient rings for the well-known skew constacyclic codes \cite{NevinsPumpluen2026codes}.

Furthermore, proper nonassociative Petit division rings $K[t;\sigma]/K[t;\sigma](t^m-a)$  over number fields have been successfully used in
  Learning with Errors (LWE), a branch of cryptography, where they constitute the ambient algebras in a nonassociative variation of LWE by Mendelsohn and Ling \cite{MendelsohnLing2025}, which outperforms previous ones in several parameters. In both these applications, it is important to understand their isomorphisms and parametrisations, and not just study them up to isotopy and only over finite fields.

\section{Preliminaries} \label{sec:prel}

\subsection{Nonassociative algebras} \label{subsec:nonassalgs}

 Let $F$ be a field. An $F$-vector space $A$ is an
\emph{algebra} over $F$ if there exists an $F$-bilinear map $A\times
A\to A$, $(x,y) \mapsto x \cdot y$, called the \emph{multiplication} of $A$.  We usually denote the multiplication of $A$ simply by the juxtaposition $xy$. Define the \emph{associator}  of $A$ by  $[x, y, z] =
(xy) z - x (yz)$.

 An algebra $A$ is called
\emph{unital} if there is an element in $A$, denoted by 1, such that
$1x=x1=x$ for all $x\in A$. We will only consider unital finite-dimensional nonzero rings and algebras.
 If $A$ is a unital associative algebra, then we denote the group of invertible elements in $A$ by $A^\times$.

 We note that every unital nonassociative algebra $A$ is canonically also a unital nonassociative ring. Conversely, every nonassociative ring can be viewed as a nonassociative algebra over its center or any subring of its center. The following definitions thus make sense for nonassociative rings $A$ as well.

 We call a nonassociative algebra or ring a \emph{proper} nonassociative algebra or ring, if it is not associative. The
 {\it center} of $A$ is defined as ${\rm C}(A)=\{x\in A\,|\, x\in \text{Nuc}(A) \text{ and }xy=yx \text{ for all }y\in A\}$.
Here,  ${\rm Nuc}(A) = \{ x \in A \, \vert \, [x, A, A] = [A, x, A] = [A,A, x] = 0 \}$ is  the {\it nucleus} of $A$, which in turn is defined as the intersection of the  {\it left, middle and right  nucleus} of $A$, defined as ${\rm
Nuc}_l(A) = \{ x \in A \, \vert \, [x, A, A]  = 0 \}$, ${\rm Nuc}_m(A) = \{ x \in A \, \vert \,
[A, x, A]  = 0 \}$ and  ${\rm Nuc}_r(A) = \{ x \in A \, \vert \, [A,A, x]  = 0 \}$, respectively.  These are  associative
subalgebras of $A$.

An algebra $A\not=0$ is called a \emph{division algebra} if for any $a\in A$, $a\not=0$, the left multiplication  with $a$, $L_a(x)=ax$, and the right multiplication with $a$, $R_a(x)=xa$, are bijective.
If $A$ has finite dimension over $F$, $A$ is a division algebra if
and only if $A$ has no zero divisors \cite[ pp. 15, 16]{Sch}. We will only look at finite-dimensional algebras.

 A \emph{semifield} is a unital nonassociative division algebra over a finite field. A semifield that is not associative is thus  called a proper semifield.

 We denote the set of not necessarily unital algebra structures on an
 $F$-vector space $V$ by ${\rm Alg}(V)$. Given $A\in {\rm Alg}(V)$, we write $x\circ_Ay$ for the product of $x,y\in V$ in the algebra $A$ instead of simply juxtaposition, when the  algebra multiplication  is not ad hoc clear from the context (subsequently, usually it will be obvious).
 For $P,Q,U\in \GL(V)$  the algebra $A^{(Q,P,U)}$ is called an
{\it isotope} of $A$, and is defined as the vector space $V$ together with the new multiplication
$$xA^{(Q,P,U)}y=U(Q(x)\circ_A P(y))$$
for $x,y\in V$.  If $A$ is a division algebra then $A^{(Q,P,U)}$ is a division algebra. Two algebras $A, A'\in {\rm Alg}(V)$ are called {\it isotopic}  if
there exist $P,Q,U\in \GL(V)$ such that $x\circ_A y=U(Q(x) \circ_{A'} P(y))$ for all $ x,y\in V.$ If $P=Q=U^{-1}$ then $A\cong A'$.
 We know that $(A^{(Q,P,U)})^{op}=(A^{op})^{(P,Q,U)}$. Note that in Section \ref{sec:isotopy} we will also use isotopies of the  
  form $U(x\circ_A y)=Q(x)\circ_{A'}P(y)$.

\subsection{Some results for Petit algebras} Let $K/F$ be a cyclic Galois extension of $F$ of degree $n$ and let $\sigma$ be a generator of $\Gal(K/F)$. Set $R=K[t;\sigma]$ and let $f\in R$.
 Then
for all $g\in R$,  there exist  uniquely
determined $r,q\in R$ with
 ${\rm deg}(r)<{\rm deg}(f)$, such that
$g(t)=q(t)f(t)+r(t),$ and  uniquely
determined $r',q'\in R$ with ${\rm deg}(r')<{\rm deg}(f)$, such
that $g(t)=f(t)q'(t)+r'(t)$.

Let ${\rm mod}_r f$ denote the remainder of right division by $f$ and
 ${\rm mod}_l f$ the remainder of left division by $f$. Define $R_m=\{h\in R\,|\, \deg(h)<m\}$.
Since the remainders are uniquely determined, the set $R_m$ canonically represents the
elements of the (left, resp. right) $R$-modules $R/Rf$ and $R/ fR$.

 The additive group $R_m$  together with the multiplication
$$u \circ_f v=uv \,\,{\rm mod}_r f$$
  is a unital nonassociative ring  denoted by $R/Rf$ or  $\mathbb{S}_f$.

We will study the case that $f(t)=t^m-a$ for some $a\in K^\times$.
We know that $\mathbb{S}_f=K[t;\sigma]/K[t;\sigma](t^m-a)$ is a proper nonassociative algebra exactly when $\sigma(a)\neq a$ or $n\nmid m$  \cite[(15)]{P66}.

For $n=m$, the algebra $K[t;\sigma]/K[t;\sigma](t^m-a)$
is called a \emph{nonassociative  cyclic algebra of degree $n$} and denoted by  $(K/F, \sigma, a)$  (\emph{cf.} \cite{S12} who studies the opposite algebras, however).  When $a\in F^\times$, the associative algebra $(K/F,\sigma, a)$ is a classical associative central simple \emph{cyclic algebra} of degree $m$ \cite[p.~19]{J96}.

 \begin{theorem}\label{thm:nuc}
 Let $A=K[t;\sigma]/K[t;\sigma](t^m-a)$ be a proper nonassociative Petit algebra. 
 Let $H=\{\tau\in {\rm Gal}(K/F)\,|\,\tau(a)=a \}$. Then $H=\langle \sigma^{s}\rangle$ for some integer $s$ such that $n=sr$. Put $E={\rm Fix}(\sigma^s)$.
 \\ (i)
  Suppose  $n\mid m$ and $\sigma(a)\not=a$. Then
 $${\rm Nuc}_r(A)=K[t^s;\sigma^s]/K[t^s;\sigma^s]((t^s)^{m/s}-a)=\bigoplus_{\{i\,|\, i<m, s\mid i \}}Kt^i $$
 has dimension $m/s$ over $K$.  In particular,
 $K\subset {\rm Nuc}(A)$.

 For $m=n$ this is the
cyclic algebra $(K/E,\sigma^s,a)$ of degree $r$  over $E$ (for $m=qn$ with $q>1$ it
cannot be cyclic of degree $r$ over $E$, since its centre is larger than $E$).
 \\ (ii) Suppose now that $n$ does not divide $m$ and put $K_0={\rm Fix}(\sigma^m)$ and $I=\{0\le i<m \,|\, \sigma^i \text{ fixes } N_{K/K_0}(a)\}$. Then
\begin{equation}\label{eq:nucr}
  {\rm Nuc}_r(A)=\bigoplus_{i\in I}K_0g_it^i
  \end{equation}
  where $g_i\in K^\times$ is a non-trivial solution of the equation $\sigma^m(g_i)\sigma^i(a)=ag_i$.
\end{theorem}

\begin{proof}
(i) Suppose first that $n\mid m$ and $\sigma(a)\not=a$. Let $E={\rm Fix}(\sigma^s)$, then $n=sr$, $a\in E$ and $[E:F]=s>1$. By \cite[Theorem 61]{Owen2021}, then $t^s\in {\rm Nuc}_r(A)$. A straightforward calculation shows that  $K\subset {\rm Nuc}_r(A)$.  Write $m=qn=qsr$. Then by \cite[Proposition 62]{Owen2021},  $s\mid m$ and so
$${\rm Nuc}_r(A)=K\oplus Kt^s\oplus \dots\oplus  Kt^{(m/s-1)s}= \bigoplus_{  \{i\,|\, i<m, s\mid i \}  } Kt^i $$
This is the Petit algebra $K[t^s;\sigma^s]/K[t^s;\sigma^s]((t^s)^{m/s}-a)$
over $E={\rm Fix}(\sigma^s)$. Since $a\in E$ and $[E:F]=s$, this is a classical associative quotient  algebra
  over $E$. The rest of the assertion is now clear.
\\ (ii)
Suppose now that $n$ does not divide $m$.
Put $K_0={\rm Fix}(\sigma^m)$.
A straightforward calculation shows that  $K_0={\rm Fix}(\sigma^m)\cap K\subset {\rm Nuc}_r(A)$  is the largest subfield of $K$ that lies in the right nucleus.

Using ${\rm Nuc}_r(S_f)=\{g(t)\in K[t;\sigma]_m\,|\, fg\in Rf\}$ and
$t^{m+i}=\sigma^i(a)t^i$, we know that for $g(t)=\sum_i^{m-1} g_it^i$, we have
\[
  {\rm Nuc}_r(A)=\{g(t)=\sum_i^{m-1} g_it^i \,|\,
  \sigma^m(g_i)\sigma^i(a)=ag_i \text{ for all } i\}.
\]
We now apply Hilbert 90 to the cyclic extension $K/K_0$ with generator $\sigma^m$, to see that the
 equation $\sigma^m(g_i)\sigma^i(a)=ag_i$ has a non-trivial solution $g_i\in K$ if and only if
$\sigma^i(N_{K/K_0}(a))=N_{K/K_0}(a)$. The solutions are given by
$K_0g_i$ for one such $g_i \in K$. Hence for  $I=\{0\le i<m \,|\, \sigma^i \text{ fixes } N_{K/K_0}(a)\}$ we get
the assertion that ${\rm Nuc}_r(A)=\bigoplus_{i\in I}K_0g_it^i.$
\end{proof}

\begin{remark} Over finite fields and for irreducible $t^m-a$, \eqref{eq:nucr} recovers the result that
${\rm Nuc}_r(\mathbb{F}_{q^n}[t;\sigma]/\mathbb{F}_{q^n}[t;\sigma](t^m-a))\cong\mathbb{F}_{q^m}$ from \cite{LS} and shows that
${\rm Nuc}_r (\mathbb{F}_{q^n}[t;\sigma]/\mathbb{F}_{q^n}[t;\sigma](t^m-a))\not\subset K$ if and only if $m\nmid n$.
\end{remark}

\begin{lemma}\label{Lem:differentextensions}
Suppose $K,K'$ are two non-isomorphic cyclic Galois extensions of a field $F$.  Then for any generators $\sigma,\sigma'$  of their respective Galois groups, and any $a\in K^\times,a'\in {K'}^\times$,
we have
$$K[t;\sigma]/K[t;\sigma](t^m-a)\not\cong K'[t;\sigma']/K'[t;\sigma'](t^m-a')$$
as nonassociative rings when  $\sigma(a)\neq a$ or $n\nmid m$.
\end{lemma}

The proof is immediate, since  isomorphic proper nonassociative rings must have the same left nuclei.

Define
$$
{\rm Aut}(K)_\sigma=\{\tau\in {\rm Aut}(K) \mid \tau \sigma = \sigma \tau\}.
$$

\begin{lemma}
(i) The $F$-stable automorphisms form the normaliser subgroup
$N_{\Aut(K)}(\langle\sigma\rangle)$, and
$\Aut(K)_\sigma=\Cent_{\Aut(K)}(\sigma)$ is a subgroup of $N_{\Aut(K)}(\langle\sigma\rangle)$, with equality when
$\Aut(K)$ is abelian.
\\ (ii) 
Suppose that ${\rm Aut}(K)$ is abelian.
Then the elements of the group ${\rm Aut}(K)_\sigma$ are precisely the automorphisms of $K$ that preserve the subfield $F$ (that is, that restrict to automorphisms of $F$).
\end{lemma}

\begin{proof}
 (i)  Since $\Aut(K)_\sigma$ is the centralizer subgroup $\Cent_{\Aut(K)}(\sigma)$, it is a group.  As $\sigma$ generates $\Gal(K/F)$, it follows that for any $a\in K$, we have $a\in F$ if and only if $\sigma(a)=a$.  Now suppose $\tau\in \Aut(K)_\sigma$.  Then for any $a\in K$, we have
$$
\tau(a) = \tau(\sigma(a))=\sigma(\tau(a)),
$$
and therefore $a\in F$ if and only if $\tau(a)\in F$, so $F$ is preserved by $\tau$. Conversely, if
$\tau$ preserves $F$ then $\tau\sigma\tau^{-1}\in\Gal(K/F)=\langle\sigma\rangle$,
so $\tau\sigma\tau^{-1}=\sigma^{j}$ for some positive integer $j$.
\\ $(ii)$ follows from $(i)$.
\end{proof}

 For any $i\in \mathbb{N}$, $\tau\in \Aut(K)$ and $\beta\in K$, write
$$
N_i^{\tau}(\beta) = \prod_{j=0}^{i-1}\tau^j(\beta).
$$
Note that $\beta \tau(N_i^\tau(\beta))=N_{i+1}^\tau(\beta)$ and thus
\begin{equation}\label{E:Normrelation}
N_{i+j}^\tau(\beta)= N_{i}^\tau(\beta) \cdot \tau^{i}(N_j^\tau(\beta))
\end{equation}
for all $i,j\geq 0$.  If ${\rm Gal}(K/F)=\langle \sigma\rangle$ and $\sigma$ has order $n$, then  $N_{n}^\sigma$ is the norm  $N_{K/F}$ of the Galois field extension $K/F$.
 Set $F^\tau={\rm Fix}(\sigma,\tau)$ to be their common fixed field.

When $\Aut(K)$ is commutative, then
for  two proper  Petit algebras $R/Rf$ and $R/Rg$ with  $n \geq m-1$, every isomorphism $G:R/Rf\rightarrow R/Rg$ is induced by a \emph{monomial isomorphism} $G: K[t;\sigma]\rightarrow K[t;\sigma]$ \emph{of degree one}, that is $G(t)=\alpha t$  for some $\alpha\in K^\times$, which restricts to an automorphism $\tau\in {\rm Aut}(K)$.
We write $G_{\tau,\alpha}=G$.
Explicitly, this means
$$G_{\tau,\alpha}(\sum_{i=0}^{m-1}d_it^i)=\sum_{i=0}^{m-1}\tau(d_i) N^\sigma_{i}(\alpha) t^i$$
\cite[Theorems 4.1, 4.2]{Pum2027}.
We  write $G_{\tau,\alpha,k}$ for an isomorphism $G:R/Rf\rightarrow R/Rg$  which restricts to an automorphism $\tau\in {\rm Aut}(K)$, and that is induced by a \emph{monomial isomorphism} $G: K[t;\sigma]\rightarrow K[t;\sigma]$ \emph{of degree $k$}, that is $G(t)=\alpha t^k$  for some $\alpha\in K^\times$.

We say two  Petit rings are \emph{isometric} if there exists a monomial isomorphism  $G_{\tau,\alpha,k}$  between them and
\emph{equivalent} if $k=1$.

In the rest of the paper, we will employ the following strong classification theorem which  generalizes \cite[Theorem 4.9]{CB}.

\begin{theorem}\label{T:classifyisomorphism}
\begin{enumerate}[(i)] \item  \cite[Corollary 6.2]{NevinsPumpluen2025submitted}
If $n \nmid m$ or if $n=m$ and $a\in K\smallsetminus  F$, then $K[t;\sigma]/K[t;\sigma](t^m-a)$ is a proper nonassociative Petit algebra, and any isomorphism
 $$G:K[t;\sigma]/K[t;\sigma](t^m-a)\to K[t;\sigma]/K[t;\sigma](t^m-b)$$
   whose restriction to $K$ is given by some automorphism $\tau\in {\rm Aut}(K)_\sigma$ must be of the form $G=G_{\tau,\alpha}$
  for some $\alpha\in K^\times$.
\item
Let $\alpha\in K^\times$ and $\tau \in {\rm Aut}(K)_\sigma$. Then
   $$G_{\tau, \alpha}: K[t;\sigma]/K[t;\sigma](t^m-a)\to K[t;\sigma]/K[t;\sigma](t^m-b)$$
    is an isomorphism of $F^\tau$-algebras if and only if
\begin{equation}  \label{equivrelation}
\tau(a) = N_{m}^{\sigma}( \alpha)b.
\end{equation}
\end{enumerate}
\end{theorem}

Note that in the setting of the theorem, $G_{\tau, \alpha}: K[t;\sigma]/K[t;\sigma](t^m-a)\to K[t;\sigma]/K[t;\sigma](t^m-b)$ is an isomorphism of nonassociative  rings, and an isomorphism of $F^\tau$-algebras.

We know that $z\in K^\times$ and $t-z$ right divides $t^m-a$ if and only if
$$N_m^\sigma(z)=a$$
 \cite{J96}. In that case, $t^m-a$ is reducible.
 When $m\geq 5$ is a prime and $F$ contains a primitive $m$th root of unity we know more.

\begin{theorem} \cite[Theorem 3.11]{CB}
Suppose either that $m=2$ or $3$, or that $m\geq 5$ is a prime and that $F$ contains a primitive $m$th root of unity.
Then  $$K[t;\sigma]/K[t;\sigma](t^m-a)\text{ is a division algebra}$$
 if and only if
$$N_m^\sigma(z)\not=a$$
for all $z\in K$.
\end{theorem}

\begin{corollary}\label{cor:div}
Suppose either that $m=2$ or $3$, or that $m\geq 5$ is a prime and $F$ contains a primitive $m$th root of unity.
Let $a\in K^\times$.
  If $N_{K/F}(a)\not\in N_{K/F}(K^\times)^m$, then $K[t;\sigma]/K[t;\sigma](t^m-a)$ is a division algebra.
\end{corollary}

\begin{proof}
Suppose there exists some $z\in K$, such that $N_m^\sigma(z)=a$. Apply $N_{K/F}$ to both sides of this equation, then
$N_{K/F}(N_m^\sigma(z))=N_{K/F}(z)^m=N_{K/F}(a).$
This yields the assertions.
\end{proof}

For $n=m$ we  have a stronger result.

\begin{theorem}\label{SteeleThm}  \cite[Theorem 4.4, Corollary 4.5]{S12}
Suppose that $(K/F,\sigma,a)$ is a nonassociative cyclic algebra where $K/F$ is of degree $m$.
If $a$ does not lie in a proper subfield of $K/F$, then $(K/F,\sigma,a)$ is a division algebra.
\end{theorem}

By Theorem \ref{thm:nuc}, for
$H=\{\tau\in {\rm Gal}(K/F)\,|\,\tau(a)=a \}$ there is some integer $s_0$ such that $n=s_0r$ and $H=\langle \sigma^{s_0}\rangle$. 
Then the right nucleus $${\rm Nuc}_r((K/F,\sigma, a))=(K/E,\sigma^s, a)$$
of a nonassociative cyclic algebra $(K/F,\sigma, a)$
 is an associative cyclic algebra of degree $r$ over its center which is the field $E={\rm Fix}(\sigma^{s_0})$, and $[E:F]=s_0$. This implies that  ${\rm Nuc}((K/F,\sigma, a))=K$. If  $(K/F, \sigma, a)$ is proper nonassociative and $m$ is prime, then ${\rm Nuc}_r((K/F, \sigma, a))=K$.

\begin{corollary} (\cite[Corollary 2.4]{St14})
Let $F$ be a finite field. Then $(K/F, \sigma, a)$ is  a semifield   if and only if $F(a)=K$.
\end{corollary}

\begin{proof}
Over finite fields there are no cyclic algebras which are division algebras  other than the field itself. So here the right nucleus  is a cyclic algebra of degree $r>1$,  hence has zero divisors whenever $r>1$. We have  $r>1$ if and only if
$s_0<n$, which is equivalent to saying that $a$ lies in a proper subfield of $K$.
 So this case is equivalent to the right nucleus having zero divisors and thus so does the algebra $(K/F, \sigma, a)$.
\end{proof}

\section{Different generators of $\Gal(K/F)$ yield non-isomorphic Sandler algebras} \label{sec:main}

Let $F$ be an arbitrary field such that $F$ admits a cyclic Galois extension $K$ of degree $n>2$ and let ${\rm Gal}(K/F)=\langle \sigma\rangle$. Let $f,g\in R=K[t;\sigma]$ be two skew-polynomials of degree $m$ and suppose that neither one generates a two-sided ideal.

The next proof is inspired by the arguments of the proofs of \cite[Theorem 5.1]{NevinsPumpluen2025} and 
\cite[Theorem 6.1]{NevinsPumpluen2025submitted}.

\begin{theorem}\label{thm:firstmain} (joint with M Nevins)
Suppose $\sigma_1, \sigma_2$ are two distinct generators of $\Gal(K/F)$, and let $F_0\subset F$ be any subfield such that $\Aut_{F_0}(K)$ is abelian. Suppose $n \nmid m$. Then for every $a_1,a_2\in K^\times$, there is no isomorphism of $F_0$-algebras
$$
K[t_1;\sigma_1]/K[t_1;\sigma_1](t_1^m-a_1) \rightarrow K[t_2;\sigma_2]/K[t_2;\sigma_2](t_2^m-a_2)
$$
That is, as $\sigma$ varies over the generators of $\Gal(K/F)$, for any fixed $a\in K^\times$, all of the resulting algebras $
K[t_1;\sigma]/K[t_1;\sigma](t_1^m-a)$ are non-isomorphic as $F_0$-algebras.
\end{theorem}

\begin{proof}
   As $\sigma_1$ and $\sigma_2$ are distinct generators of $\Gal(K/F)$, there exists an integer $j$ satisfying $2\leq j <n $ and $\gcd(j,n )=1$ for which
 $$
 \sigma_1 = \sigma_2^j.
 $$
 Let $a_1,a_2\in K^\times$ and set $R_{\sigma_i,a_i}=K[t_i;\sigma_i]/K[t_i;\sigma_i](t_i^m-a_i)$ for $i\in \{1, 2\}$.

Suppose now that
$\varphi:R_{\sigma_1,a_1} \to R_{\sigma_2,a_2}$
is an isomorphism of $F_0$-algebras. Since $n \nmid m$ we know that the algebras are proper nonassociative with left nucleus  $K$ so that we get $\varphi|_K\in\Aut(K)$.
  Thus the restriction of $\varphi$ to $K$ lies in the abelian group $\Aut_{F_0}(K)$,  and commutes with $\sigma_1$ and $\sigma_2$.
 Let us write
 $$
 \varphi(t_1) = \sum_{i=0}^{m-1} b_it_2^i.
 $$
 Then since for all $k\in K$, we have $t_1k = \sigma_1(k)t_1$, it follows that $\varphi(t_1k)=\varphi(\sigma_1(k)t_1)$.  Since $\varphi$ is, in particular, a homomorphism, we infer that
$$
\varphi(t_1)\varphi(k) = \sum_{i=0}^{m-1} b_it_2^i \cdot \varphi(k)= \sum_{i=0}^{m-1} b_i\sigma_2^i(\varphi(k))t_2^i,
$$
which should be equal to
$$
\varphi(\sigma_1(k))\varphi(t_1) = \varphi(\sigma_1(k))\sum_{i=0}^{m-1} b_it_2^i.
$$
Comparing coefficients, we deduce that for each $i$ such that $b_i \neq 0$, we must have that for all $k\in K$, $$\sigma_2^i(\varphi(k))=\varphi(\sigma_1(k)).
$$
Since these automorphisms of $K$ commute, we deduce $\sigma_2^i = \sigma_1$ whence $i\equiv j \mod n $.
Let $D$ represent the set of degrees of the terms of $\varphi(t_1)$ that may have nonzero coefficients; that is, let
$$
D = \{ 0\leq i < m \mid i\equiv j \mod n \}.
$$
Note that if $n \geq m-1$, then $D$ is at most a singleton.  If $D=\emptyset$ then no such homomorphism $\varphi$ exists, and we are done.

Suppose $D$ is nonempty.  Let
$d_{\max}$ be the largest $i\in D$ with $b_i\neq 0$. Since $j \not\equiv 1 \mod n$, we have $d_{\max}>1$.
Choose the least integer $s$ such that $d_{max}s \geq m$. Since $2\leq d_{\max}\leq m-1$ we have $2\leq s \leq m-1$.

For  all $1\le l\le m-1$ we have
$t_1^lk=\sigma_1^l(k)t_1^l$ which yields $\varphi(t_1^l)\varphi(k)=\varphi(\sigma_1^l(k))\varphi(t_1^l)$.
Comparing coefficients shows that
\begin{equation}\label{eq:res}
\text{every power } t_2^e \text{ occurring in } \varphi(t_1^l) \text{ with non-zero coefficient
satisfies } e\equiv lj\, { \rm mod}\, n.
\end{equation}
Moreover, we have $t_1^l=t_1(t_1^{l-1})$ in $R_{\sigma_1,a_1}$, hence
$\varphi(t_1^l)=\varphi(t_1)\varphi(t_1^{l-1})$. 

Recall that in $R_{\sigma_2,a_2}$ we have
$$(bt_2^d)(ct_2^e)=\begin{cases}b\sigma_2^d(c)t_2^{d+e} & \text{if } d+e<m;\\
b\sigma_2^d(c)\sigma_2^{d+e-m}(a_2)t_2^{d+e-m} & \text{if } m\leq d+e<2m.\end{cases}$$

We first show by induction that $\varphi(t_1^l)$ has degree $ld_{\max}$ for all
$0\leq l\leq s-1$. This is clear for $l=0$. Let $1\le l\le s-1$ and let $c\neq 0$ be the leading
coefficient of $\varphi(t_1^{l-1})$, which has degree $(l-1)d_{\max}$. In the product
$\varphi(t_1)\varphi(t_1^{l-1})$, the term 
$$(b_{d_{\max}}t_2^{d_{\max}})(ct_2^{(l-1)d_{\max}})
=b_{d_{\max}}\sigma_2^{d_{\max}}(c)t_2^{ld_{\max}}$$
 is non-zero, since $ld_{\max}<m$. 
 All other products $(b_dt_2^d)(y_et_2^e)$ of terms of $\varphi(t_1)$ and $\varphi(t_1^{l-1})$ satisfy $d\leq d_{\max}$ and $e\leq(l-1)d_{\max}$, with at least one inequality strict. Hence $d+e<ld_{\max}<m$, so no reduction of exponents occurs and these products have degree smaller than $ld_{\max}$.
 
Now consider $\varphi(t_1^s)=\varphi(t_1)\varphi(t_1^{s-1})$, and let $c\neq 0$ be the leading
coefficient of $\varphi(t_1^{s-1})$. Since $m\le sd_{\max}<m+d_{\max}<2m$, the product
$$(b_{d_{\max}}t_2^{d_{\max}})(ct_2^{(s-1)d_{\max}})=
b_{d_{\max}}\sigma_2^{d_{\max}}(c)\sigma_2^{sd_{\max}-m}(a_2)t_2^{sd_{\max}-m}$$
has a non-zero coefficient. We claim that the coefficient of $t_2^{d_{\max}s-m}$ in
$\varphi(t_1^s)$ is non-zero, i.e. that no other product $(b_dt_2^d)(y_et_2^e)$ of a term of
$\varphi(t_1)$ and a term of $\varphi(t_1^{s-1})$ contributes to $t_2^{d_{\max}s-m}$:

If $d+e<m$, then $(b_dt_2^d)(y_et_2^e)=b_d\sigma_2^d(y_e)t_2^{d+e}$ is a multiple of
$t_2^{d+e}$. By (\ref{eq:res}), $d\equiv j$ and $e\equiv(s-1)j \mod n$, so
$d+e\equiv sj \mod n$. But $d_{\max}s-m\equiv sj-m\not\equiv sj \mod n$, since $n\nmid m$.
Hence $d+e\neq d_{\max}s-m$.

If $d+e\geq m$, then $(b_dt_2^d)(y_et_2^e)=b_d\sigma_2^d(y_e)\sigma_2^{d+e-m}(a_2)
t_2^{d+e-m}$ is a multiple of $t_2^{d+e-m}$, and $d+e-m=d_{\max}s-m$ means $d+e=d_{\max}s$.
Since $d\le d_{\max}$ and $e\le(s-1)d_{\max}$, this forces $d=d_{\max}$ and $e=(s-1)d_{\max}$.

Hence we may write
$$\varphi(t_1^s)=\lambda t_2^{d_{\max}s-m}+q(t_2),$$
where $\lambda\in K^\times$ and there are no monomial terms of degree $d_{\max}s-m$
appearing in $q(t_2)$. By (\ref{eq:res}) with $l=s\leq m-1$, we obtain
$d_{\max}s-m\equiv sj \,{\rm mod}\,  n$. But $d_{\max}\equiv j \,{\rm mod}\, n$, so this implies $n\mid m$,
contradicting the hypothesis.

Thus no such homomorphism $\varphi$ can exist, which yields the statement of the theorem.
\end{proof}

We can also settle the additional case that  $n=m$ and $a_1\in K\setminus F$, extending the result of \cite[Theorem 5.1]{NevinsPumpluen2025} to a distinction result for $F_0$-algebras.

\begin{theorem}\label{thm:main1}
 Suppose $\sigma_1, \sigma_2$ are two distinct generators of $\Gal(K/F)$, and let $F_0\subset F$ be any subfield such that $\Aut_{F_0}(K)$ is abelian. Suppose $n=\de$.  Then for every $a_1 \in K^\times$ and $a_2\in K\smallsetminus F$, there is no isomorphism of $F_0$-algebras
    $$
    K[t_1;\sigma_1]/K[t_1;\sigma_1](t_1^\de-a_1) \to K[t_2;\sigma_2]/K[t_2;\sigma_2](t_2^\de-a_2).
    $$
\end{theorem}

The asymmetry of the statement in $a_i$ is for emphasis: when $a_1\in F^\times$, the algebra $R_{\sigma_1,a_1}$ is associative, so the lack of isomorphism is clear.  Our theorem reveals that outside of that setting, no isomorphisms between Petit algebras of this form, for different $\sigma$, are possible.

\begin{proof}
When $a_1\in F^\times$, we are done, so assume $a_1\in K\setminus F$.
    As in the proof of Theorem~\ref{thm:firstmain}, let $1<j<n$, with $\gcd(j,n)=1$, be such that $\sigma_1=\sigma_2^j$. Then if $\varphi:R_{\sigma_1,a_1}\to R_{\sigma_2,a_2}$ is an isomorphism, all the nonzero monomial terms of $\varphi(t_1)$ must have degree $i\equiv j\mod n$.  Therefore since $n=m$ we must have $\varphi(t_1)=kt_2^j$ for some $k\in K^\times$.  Since $j>1$, the minimal value $s$ for which $sj\geq m$ satisfies $2\leq s \leq m-1$, and thus $(t_2^{j})^{s}= t_2^{js-\de}a_2$ in $R_{\sigma_2,a_2}$.
We may therefore directly compute that
$$
\varphi(t_1)^s = (kt_2^j)^{s-1}(kt_2^j) = \prod_{i=0}^{s-1} \sigma_2^{ij}(k)t_2^{js-\de}a_2 =  \sigma_2^{js-\de}(a_2)\prod_{i=0}^{s-1} \sigma_2^{ij}(k)t_2^{js-\de}.
$$
Since $n=\de$, we have $\sigma_2^{js-\de}=\sigma_2^{js}$, whence this simplifies to
$$
\varphi(t_1^s)=\sigma_2^{js}(a_2)\prod_{i=0}^{s-1} \sigma_2^{ij}(k)t_2^{js-\de}.
$$
We now compare $\varphi(t_1)\varphi(t_1^s)$ and $\varphi(t_1^s)\varphi(t_1)$, which must be equal since $\varphi$ is an isomorphism and $s\leq\de-1$.  On  the one hand we have
$$
\varphi(t_1^s)\varphi(t_1) = \sigma_2^{js}(a_2)\prod_{i=0}^{s-1} \sigma_2^{ij}(k)t_2^{js-m}\cdot kt_2^j
= \sigma_2^{js}(a_2) \prod_{i=0}^{s} \sigma_2^{ij}(k)t_2^{js-m}t_2^j
$$
while on the other
$$
\varphi(t_1)\varphi(t_1^s)=kt_2^j\cdot \sigma_2^{js}(a_2)\prod_{i=0}^{s-1} \sigma_2^{ij}(k)t_2^{js-m}
= \sigma_2^{j+js}(a_2)\prod_{i=0}^{s} \sigma_2^{ij}(k) t_2^jt_2^{js-\de}.
$$
We infer from equating the coefficients that
$\varphi(t_1)\varphi(t_1^s)$ gives
\[
  \sigma_2^{js}(a_2)\prod_{i=0}^{s}\sigma_2^{ij}(k)
  =\sigma_2^{js+j}(a_2)\prod_{i=0}^{s}\sigma_2^{ij}(k),
\]
hence $\sigma_2^{js}(a_2)=\sigma_2^{js+j}(a_2)$, i.e.
$a_2=\sigma_2^{j}(a_2)=\sigma_1(a_2)$.  Thus $a_2\in F^\times$, contradicting our hypothesis.
\end{proof}

In both the proofs of Theorems \ref{thm:firstmain} and \ref{thm:main1}, the commutativity of the group $\Aut_{F_0}(K)$ is used in the proofs only in the step where we conclude that $\sigma_2^i=\sigma_1$ from
$\sigma_2^i\tau=\tau\sigma_1$ for $\tau=\varphi|_K\in \Aut_{F_0}(K)$ (which holds for every $i$ with $b_i\neq 0$). Without it, set
$\rho=\tau\sigma_1\tau^{-1}$; then $\rho=\sigma_2^{j'}$ for a unit $j'$ modulo
$n$, and $\tau(F)=F$ is true automatically; all nonzero $b_i$ have $i\equiv j'\,{\rm mod}\, n$, and
the rest of the argument runs verbatim with $j'$ in place of $j$. In Theorem
\ref{thm:firstmain}, then Equation \ref{eq:res} holds with $j'$ in place of $j$,
and  in Theorem \ref{thm:main1} we obtain
$a_2\in\Fix(\sigma_2^{j'})=F$.

We conclude that when our base field is not finite, it can indeed happen that distinct generators of ${\rm Gal}(K/F)$ can produce isomorphic algebras, however, this is not the case when we work with semifields.
 Thus we get our first Main Theorem which is important when we are working with fields which are not finite and hence may have automorphism groups which are not abelian (which is the case when when we use these algebras over $\mathbb{Q}$ for LWE, for instance).

\begin{theorem}\label{thm:mainmain}
Let $n\nmid m$, or let $n=m$ and $a_2\in K\setminus  F$. Every ring isomorphism
$$H:K[t_1;\sigma_1]/K[t_1;\sigma_1](t_1^\de-a_1) \to K[t_2;\sigma_2]/K[t_2;\sigma_2](t_2^\de-a_2)$$
such that $H$ is $F_0$-linear where $F_0\subset F$ is a subfield of $K$,
 satisfies
$H|_K\circ\sigma_1=\sigma_2\circ H|_K$. In particular
$\sigma_1=\sigma_2$ whenever $H|_K$ commutes with $\sigma_2$, e.g. when
$\Aut_{F_0}(K)$ is abelian.
\end{theorem}

\begin{example}
Let $F=\mathbb{Q}(i)$, $K=\mathbb{Q}(i,\sqrt[4]{2})$,
$\sigma:\sqrt[4]{2}\mapsto i\sqrt[4]{2}$, and let $\tau$ denote the complex conjugation. Then
$\tau(F)=F$ and $\tau\sigma\tau^{-1}=\sigma^{-1}\neq\sigma$.

The map
$G_{\tau,1}:K[t;\sigma]\to K[t;\sigma^{-1}],$ $G_{\tau,1}(\sum k_it^i)=\sum\tau(k_i)t^i$ is a ring isomorphism since we have $\tau\sigma=\sigma^{-1}\tau$. It canonically
induces a $\mathbb{Q}$-algebra isomorphism
$$K[t;\sigma]/K[t;\sigma](t^m-a) \to K[t;\sigma^{-1}]/K[t;\sigma^{-1}](t^m-\tau(a))$$
 for every integer $m>1$ and all $a\in K^\times$. So
if $\Aut_{F_0}(K)$ is not abelian, then distinct generators can give
isomorphic algebras, and the $F_0$-isomorphism classes are parametrised by
orbits of $\Aut_{F_0}(K)$ on the pairs via
$(\sigma,a)\mapsto(\tau\sigma\tau^{-1},\tau(a))$, with $\tau\in \Aut_{F_0}(K)$, combined with the usual norm
relation that restricts the admissible choices on $\alpha\in K^\times$.
\end{example}

\begin{corollary}\label{cor:Euler}
Suppose that $\Aut(K)$ is abelian.
Let either $n \nmid m$ or, alternatively, let  $n =m $ and suppose that $a\in K\setminus  F$.
Let $\varphi$ be the Euler function. Then there are $\varphi(n)$ different families  of proper Petit algebras of type $K[t;\sigma_i]/K[t;\sigma_i](t^m-a)$; one family for each choice of generator $\sigma_i$ of the cyclic Galois group $\Gal(K/F)$. Algebras in two different families are not isomorphic as nonassociative rings.
\end{corollary}

\begin{corollary}\label{cor:Eulerfinite}
Suppose $F$ is a finite field and $K/F$ a finite field extension of degree $n$.
Let either $n \nmid m$ or, alternatively, let  $n =m $ and suppose that $a\in K\setminus  F$.
Let $\varphi$ be the Euler function. Then there are $\varphi(n)$ different families of proper Petit semifields of type $K[t;\sigma_i]/K[t;\sigma_i](t^m-a)$; one family for each choice of generator $\sigma_i$ of the cyclic Galois group $\Gal(K/F)$. Algebras in two different families are not isomorphic as semifields.
\end{corollary}

These results leave open the case that 
$n\mid m$, where there remain also many open questions about isomorphisms, even in the simple case that $\sigma_1=\sigma_2$ \cite{NevinsPumpluen2025submitted, NevinsPumpluen2026codes}.

\section{Isotopic Sandler algebras are isomorphic when $\sigma_1=\sigma_2$} \label{sec:isotopy}

The proof of  \cite[Theorem 9]{Sandler1962} can be generalized to arbitrary base fields and yields our Second Main Theorem. Sandler works over finite fields, and uses the assumption that $n =m$ and $\sigma_1=\sigma_2$. A closer look at the proof shows that the restrictions that the fields are finite are not necessary when  ${\Aut}(K)$ is abelian, and neither is the restriction that $n=m$.
Note that  Sandler's algebras are the opposite algebras of the Petit algebras $K[t;\sigma]/K[t;\sigma](t^m-a)$. We thus prove some results for these, to keep close to his original proofs.

We break up the proof and start with a generalization of some intermediate results in \cite{Sandler1962}.
The following  lemma is a straightforward generalization of  \cite[Lemma, p. 243]{Sandler1962}.
Note that $(F_{\sigma^{-1},1},L_t,L_t)$ is the autotopism called $E_1$ in \cite[p.~256]{Sandler1962}.

\begin{lemma} \label{le:nuc}
 Let $(Q,P,U)$ be an isotopy between two proper nonassociative algebras $A_1$ and $A_2$ over $F$ with middle nucleus $K$
 where $U(x\circ_1 y)=Q(x)\circ_2 P(y)$
  for all $x,y\in A_1$. Then for every
$a\in {\rm Nuc}_m(A_1)=K$ there is $c\in {\rm Nuc}_m(A_2)=K$ such that
$$PL_a=L_cP,$$ 
and $a\mapsto c$ is a
ring isomorphism ${\rm Nuc}_m(A_1)\cong {\rm Nuc}_m(A_2)$. In particular, if ${\rm Nuc}_m(A_1)={\rm Nuc}_m(A_2)=K$, then $c=V(a)$ for some $V\in{\rm Aut}(K)$.
\end{lemma}

\begin{proof}
 Analogously as  in \cite[Lemma, p. 243]{Sandler1962} we can show that for any $a\in {\rm Nuc}_m(A_1)=K$ there is $c\in {\rm Nuc}_m(A_2)=K$ such that $PL_a=L_cP$ where now $L_a$ is left multiplication in $A_1$ and $L_c$  the one in $A_2$.
\end{proof}

 In the remainder of this section we assume that $K/F$ is a cyclic Galois field extension of degree $n$ and $\sigma_1$, $\sigma_2$ are two generators of ${\rm Gal}(K/F)$. Let  $\sigma_1$, $\sigma_2$, $\sigma$ be  generators of ${\rm Gal}(K/F)$, so that
$\sigma_1=\sigma_2^{r}$ for some integer $r$ coprime to $n$.  We assume  that ${\Aut}(K)$ is abelian.

\begin{lemma}\label{le:shifty}
Let $n=m$, $a\in K^\times$ and $A=(K[t;\sigma]/K[t;\sigma](t^m-a))^{op}$. For $\tau\in{\rm Aut}(K)$ define
$F_{\tau,1}(\sum_{i=0}^{m-1}t^ix_i)=\sum_{i=0}^{m-1}t^i\tau(x_i).$ Then
$$t\circ_A(x\circ_A y)=F_{\sigma^{-1},1}(x)\circ_A (t\circ_A y)$$
for all $x,y\in A$, i.e. $(F_{\sigma^{-1},1},L_t,L_t)$ is an autotopism of $A$.
\end{lemma}

\begin{proof}
Since  $A$ is the opposite of a Petit algebra, we have $(t^ib)(t^jc)=t^it^j\sigma^j(b)c$ for  all integers $i,j$ with $0\leq i,j\leq m-1$
and $b,c\in K$, where $t^it^j=t^{i+j}$ if  $ i+j<m,$ and $ t^it^j=t^{i+j-m}a$ if $ i+j\geq m .$
In particular, $t(t^kd)=t^{k+1}d$ for $k\leq m-2$ and $t(t^{m-1}d)=ad$.

It suffices to work with $x=t^ib$ and
$y=t^jc$ with $b,c\in K$. Put $w=\sigma^j(b)c$, so that $x\circ y=t^it^jw$.

Suppose that  $j\leq m-2$. Then $t y=t^{j+1}c$, and
$$F_{\sigma^{-1},1}(x)(t y)=(t^i\sigma^{-1}(b))(t^{j+1}c)
=t^it^{j+1}\sigma^{j+1}(\sigma^{-1}(b))c=t^it^{j+1}w .$$
If $i+j\le m-2$, then $t(x y)=t(t^{i+j}w)=t^{i+j+1}w=t^it^{j+1}w$.
If $i+j=m-1$, then $t(x y)=t(t^{m-1}w)=aw=t^it^{j+1}w$.
If $i+j\ge m$, then $x y=t^{i+j-m}aw$ with $i+j-m\leq m-3$, hence
$t(x y)=t^{i+j+1-m}aw=t^it^{j+1}w$.

Suppose now that $j=m-1$. Then $t y=ac$, and
$$F_{\sigma^{-1},1}(x)(t y)=(t^i\sigma^{-1}(b))(ac)=t^i\sigma^{-1}(b)ac .$$
If $i=0$, then $t(x y)=t(t^{m-1}w)=aw=a\sigma^{m-1}(b)c$.
If $i\ge1$, then $xy=t^{i-1}aw$, hence $t(x y)=t^{i}a\sigma^{m-1}(b)c$.
Since $n=m$, we have $\sigma^{m-1}=\sigma^{-1}$, so both expressions agree with
$t^i\sigma^{-1}(b)ac$.
Thus $(F_{\sigma^{-1},1},L_t,L_t)$ is an autotopism of $A$. Note that here $L_t$ is bijective since we assume that  $a\in K^\times$.
\end{proof}

 Let now $m\geq 2$ and let $A_1=(K[t;\sigma_1]/K[t;\sigma_1](t^m-a_1))^{op}$ and
$A_2=(K[t;\sigma_2]/K[t;\sigma_2](t^m-a_2))^{op}$ be  two proper nonassociative division algebras such that one of the following holds:
 $n > m$ and  $\sigma_i(a_i)\not=a_i$, or we have
 $n=m$ and $a_i$ lies in no proper subfield of $K$.
 
  We first generalize \cite[Theorem 4]{Sandler1962}. We sketch the important steps in the proofs, which to large parts is analogous to the one presented in \cite{Sandler1962}. Note that Theorem \ref{thm:isot} considers also distinct generators, unlike the original result by Sandler, so is much broader.

\begin{theorem} \label{thm:isot}
 Let $(Q,P,U)$ be an isotopy between  $A_1$ and
$A_2$, that is $U(x\circ_1 y)=Q(x)\circ_2 P(y)$ for all $x,y\in A_1$.
 Then there exists some $V\in {\rm Aut}(K)$ such that $P=(H_{i,j})_{i,j}$ is a monomial matrix, where
$$H_{i,j}(y)=h_{i,j}V \sigma_2^{(i-1)+r(1-j)}(y)$$
 with $h_{i,j}\in K$.
 The non-zero entries of $P$ lie on the transversal
$$\{(k_j,j)\,|\, 1\le j\leq m\},\quad k_j\equiv k_1+r(j-1)\, {\rm mod}\, n .$$
\end{theorem}

\begin{proof}
 We write elements of $A_i$  as $y=\sum_{j=1}^{m} t^{j-1}y_j$ with $y_j\in K$, and will compose maps from right to left.
 We write
 $$P(\sum_{j=1}^m t^{j-1}y_j)=\sum_{i=1}^m t^{i-1}\sum_{j=1}^m H_{ij}(y_j),$$
so that we have $P=(H_{i,j})_{i,j}$ (acting on column vectors), where every $H_{i,j}$ is an $F_0$-linear transformation of $K$, and $F_0$ is the prime field of $F$.

For all $a\in K$, left multiplication by $a$ in $A_i$ is given by
$$L_a={\rm diag}(R(a),R(\sigma_i(a)),\dots,R(\sigma_i^{m-1}(a)))$$
 where $R(b)$ is right multiplication with $b$ in $K$ (we follow Sandler's notation here to stay close to his proof).

By Lemma \ref{le:nuc} there is $V\in{\rm Aut}(K)$ such that $PL_a=L_{V(a)}P$
for all $a\in K$, the middle nucleus of both algebras. Put $U_{i,j}= V\circ \sigma_2^{(i-1)+r(1-j)}$. Comparing entries, this means
$$H_{ij}\circ R(\sigma_1^{\,j-1}(a)) = R(\sigma_2^{\,i-1}(V(a)))\circ H_{ij}$$
for all $a\in K$ and all $\ 1\leq i,j\leq m$.

By \cite[Lemma 2]{Albert} (which holds in our setting) this forces
$H_{i,j}=R(h_{i,j})\circ U_{i,j}$, i.e. $H_{i,j}(y)=h_{ij}\,U_{i,j}(y)$, with $h_{ij}\in K$ and
$$\sigma_2^{i-1}\circ V\circ \sigma_1^{1-j}
= \sigma_2^{i-1}\circ V\circ\sigma_2^{r(1-j)}
= V\circ \sigma_2^{(i-1)+r(1-j)}=U_{i,j},$$
using that ${\rm Aut}(K)$ is abelian. Here, $R(h_{i,j})$ denotes the multiplication with $h_{i,j}$ in $K$.

Let $Q(1)=\sum_{k=0}^{m-1}t^ku_k$ for some $u_i\in K$ and
$Q(t)=\sum_{k=0}^{m-1}t^k\alpha_k$ for some $\alpha_i\in K$.

From $U(x\circ_1 y)=Q(x)\circ_2 P(y)$ we see that $UL_x=L_{Q(x)}P$ for all
$x\in A_1$. Setting $x=1$ and using $L_1={\rm id}_{A_1}$ implies
$$L_{Q(z)}P=L_{Q(1)}PL_z \quad\text{for all } z\in A_1,$$
analogously to \cite[(19)]{Sandler1962} where $L_{Q(z)}$ denotes the left
multiplication by $Q(z)\in A_2$ and $L_z$ the left multiplication by $z\in A_1$.
Putting $z=t$ yields
a system of linear equations analogous to \cite[(30), (31)]{Sandler1962}: Write
$ y=\sum_{k=1}^m t^{k-1}y_k$, recall we compose maps from right to left, and let $R(b)$ denote multiplication by $b$ in $K$.

 Analogously to Sandler's  \cite[(3)–(5)]{Sandler1962} (but now with $\sigma_2=S$ and his $\delta=a_2$), we know that
 for all $w=\sum_s t^sw_s\in A_2$, we have
$$(t^sw_s)\circ_2(t^{l-1}y_l)=t^st^{l-1}\sigma_2^{l-1}(w_s)y_l ,$$
where $t^st^{l-1}=t^{s+l-1}$ if $s+l-1<m$, and $t^st^{l-1}=t^{s+l-1-m}a_2$ otherwise. Hence we get
$$
(L_w)_{i,l}=R(\varepsilon_{il}\,\sigma_2^{\,l-1}(w_{i-l})),\qquad
\varepsilon_{i,l}=\begin{cases}1,& l\le i,\\ a_2,& l>i,\end{cases}
$$
with the index $i-l$ of $w$ read modulo $m$.

We have $t\circ_1(t^{k-1}y_k)=t^ky_k$ for $k<m$, and $t\circ_1(t^{m-1}y_m)=a_1y_m$. So the $(k+1,k)$-entry is $(L_t)_{k+1,k}=id$,
the $(1,m)$-entry is $(L_t)_{1,m}=R(a_1)$, and all other entries are $0$. This gives
$$(PL_t)_{l,j}=H_{l,j+1}\text{ for all } j<m \text{ and } (PL_t)_{l,m}=H_{l,1}\circ R(a_1).$$
Comparing the $(i,j)$-entries of $L_{Q(t)}P=L_{Q(1)}PL_t$ gives
\begin{equation}\label{eq:31j}
\sum_{l=1}^{m} R(\varepsilon_{i,l}\,\sigma_2^{\,l-1}(\alpha_{i-l}))H_{l,j}
=\sum_{l=1}^{m} R(\varepsilon_{i,l}\,\sigma_2^{\,l-1}(u_{i-l}))H_{l,j+1}
\quad (1\leq j\leq m),
\end{equation}
for all $1\leq i\le m$,
where the indices of $\alpha$ and $u$ are taken modulo $m$. Here,
$\varepsilon_{i,l}=1$ for $l\leq i$, and $\varepsilon_{i,l}=a_2$ for $l>i$.

For $j=m$,  $H_{l,j+1}=H_{l,1}R(a_1)$, so that
\begin{equation}\label{eq:31m}
\sum_{l=1}^{m} R(\varepsilon_{i,l}\,\sigma_2^{l-1}(\alpha_{i-l}))H_{l,m}
=\sum_{l=1}^{m} R(\varepsilon_{i,l}\,\sigma_2^{l-1}(u_{i-l}))H_{l,1}R(a_1).
\end{equation}
After substituting, we obtain
$H_{lj}=R(h_{l,j})\circ V\circ\sigma_2^{\,(l-1)+r(1-j)}.$
Therefore, for $j<m$ and all $y\in K$, we have
$$
\sum_{l=1}^{m}\varepsilon_{i,l}\sigma_2^{l-1}(\alpha_{i-l})h_{l,j}\;V\sigma_2^{(l-1)+r(1-j)}(y)
=\sum_{l=1}^{m}\varepsilon_{i,l}\sigma_2^{l-1}(u_{i-l})h_{l,j+1}\;V\sigma_2^{(l-1)-rj}(y),
$$
and for $j=m$ we have
$$
\sum_{l=1}^{m}\varepsilon_{i,l}\sigma_2^{l-1}(\alpha_{i-l})h_{lm}V\sigma_2^{(l-1)+r(1-m)}(y)
=\sum_{l=1}^{m}\varepsilon_{i,l}\sigma_2^{l-1}(u_{i-l})h_{l,1}V\sigma_2^{l-1}(a_1)V\sigma_2^{l-1}(y).
$$
 We now compare the coefficients,
using the independence of distinct automorphisms, to show that $P$ is monomial, similar to what Sandler did with Equation \cite[(32)]{Sandler1962}.

Arguing as in \cite[p.~247]{Sandler1962}, exactly one entry in the first column
of $P$ is non-zero, i.e. exactly one $h_{i1}$, $1\leq i\leq m$, is non-zero.
(Note that Sandler's first row corresponds to our first column, as he uses row vectors.)
Indeed, put $z=P(1)=\sum_{i=1}^m t^{i-1}h_{i1}$. By Lemma \ref{le:nuc}, we have
$P(a)=V(a)\circ_2 z$ for all $a\in K={\rm Nuc}_m(A_1)$.

For all
$b\in K={\rm Nuc}_r(A_1)$, $PR_b=R_{b'}P$ for some $b'\in{\rm Nuc}_r(A_2)=K$,
where the map $b\mapsto b'$ is bijective; hence we get $P(b)=z\circ_2 b'$. Thus
$\{c\circ_2 z \,|\, c\in K\}=P(K)=\{z\circ_2 d \,|\, d\in K\}$. Since
$c\circ_2 z=\sum_i t^{i-1}\sigma_2^{i-1}(c)h_{i,1}$ and $z\circ_2 d=\sum_i t^{i-1}h_{i,1}d$,
two non-zero coefficients $h_{i,1},h_{j,1}$ with $i<j$ would force
$\sigma_2^{j-i}=id$, contradicting $1\le j-i\leq m-1<n$ and ${\rm ord}(\sigma_2)=n\geq m$. Note that in this step we indeed need that
 $m\le n$.

As $P(1)\neq 0$, exactly one $h_{i,1}$ is non-zero.

Arguing as in \cite[p.~246]{Sandler1962}, since exactly one entry $h_{k_1, 1}$ in the
first column of $P$ is non-zero, then exactly one entry in each row and each column
of $P$ is non-zero.
Indeed, if column $j$ with $j<m$ has exactly one non-zero entry $h_{k_j j}$, we can compare the
coefficients of the pairwise distinct automorphisms $V\sigma_2^{(l-1)-rj}$
with $1\leq l\leq m$, which are  distinct since $m\leq n$, in the $(i,j)$-equations of
$L_{Q(t)}P=L_{Q(1)}PL_t$. The left-hand side only involves the automorphism
$V\sigma_2^{(k_j-1)+r(1-j)}$, so $\varepsilon_{il}\sigma_2^{l-1}(u_{i-l})h_{l,j+1}=0$ for
all $i$ and all $l\not\equiv k_j+r \, {\rm mod}\, n$. Choosing $i$ with $u_{i-l}\neq 0$ (possible
as $Q(1)\neq 0$) and using $a_2\neq 0$ gives $h_{l,j+1}=0$. Since $P$ is injective,
column $j+1$ has exactly one non-zero entry, in row $k_{j+1}\equiv k_j+r$.
The rows $k_j$ are pairwise distinct since $\gcd(r,n)=1$ and $|j-j'|<m\leq n$. Thus the indices of the non-zero entries all lie in the set
$$\{(k_j,j)\,|\, 1\leq j\leq m\},\quad k_j\equiv k_1+r(j-1) \, {\rm mod}\, n .$$

For $r=1$ and $n=m$ this is \cite[(32)]{Sandler1962}.
\end{proof}

Sandler lets maps act on the right and on row vectors, so his $(i,j)$-entry
corresponds to our $(j,i)$-entry, and his entry $S^{\,j-i}VR(h_{ij})$ is our entry
$R(h_{ji}) V\sigma^{\,j-i}$. For $\sigma_1=\sigma_2$ and $a_1=a_2$, the above is thus consistent with Sandler's Equation \cite[(29)]{Sandler1962}.

\begin{theorem} \label{thm:isotii} 
 Let $(Q,P,U)$ be an isotopy between  $A_1$ and
$A_2$, that is $U(x\circ_1 y)=Q(x)\circ_2 P(y)$ for all $x,y\in A_1$.
   Let $\sigma=\sigma_1=\sigma_2$. 
   \\ (i) Let $m\geq 3$. Then $Q(1)=u_0\in K^\times$  and there exists some $V\in {\rm Aut}(K)$ such that
\[ P=\left [\begin {array}{cccccc}
R(h_{1,1})V &    R(h_{1,2})  V \sigma^{-1}    &  \cdots & R(h_{1,m})V\sigma^{-m+1}\\
R(h_{2,1})V \sigma   & R(h_{2,2})V  &  \cdots &    R(h_{2,m})V  \sigma^{-m+2}\\
  R(h_{3,1})V  \sigma^{2}  &  R(h_{3,2}) V\sigma       &   \cdots & R(h_{3,m})V \sigma^{-m+3}\\
 \vdots     &  \vdots     &  \ddots      &  \vdots\\
   R(h_{m,1}) V  \sigma^{m-1}    &  R(h_{m,2}) V   \sigma^{m-2}  &  \cdots    & R(h_{m,m})V\\
\end {array}\right ]
 \]
 where $R(h_{i,j})$ denotes the right multiplication with  $h_{i,j}\in K$.
Moreover, if
$$P=\diag(R(h_1)V, R(h_2)V,\cdots, R(h_m)V),$$
 then
 $$h_{i+1}=h_1\,N_i^{\sigma}(h_2h_1^{-1}).$$
 In particular, when $h_1=1$ then $h_{k+1}=N_k^\sigma(h_2)$.
\\ (ii) Let $a_1=a_2$.  The subgroup of autotopisms $(Q,P,U)$ whose $P$ have the form
$$P={\rm diag}(R(h_1) V,\dots,R(h_m) V)$$
 is a normal
subgroup of the autotopism group.
\end{theorem}

\begin{proof}
 $(i)$  We keep using the notation from the proof of Theorem \ref{thm:isot}.

We will argue as in \cite[p.~252 ff.]{Sandler1962} to show that if $Q(1)=
u_0+t u_1+\dots+u_{m-1} t^{m-1}$ for $u_i\in K$ then $u_1=\dots u_{m-1}=0$. For this we need to be able to assume w.l.o.g. that $P(1)=1$.

When $n>m$ then $k_1=1$, i.e. $P(1)=h_1\in K^\times$. 
For $c\in K={\rm Nuc}_r(A_2)$,  $({\rm id},R_c,R_c)$ is an autotopism of $A_2$, because $(x\circ_2y)\circ_2c=x\circ_2(y\circ_2c)$. This induces the isotopy $(Q,R_cP,R_cU)$.
For $c=h_1^{-1}$ we get $R_cP(1)=1$.
By replacing $(Q,P,U)$ with $(Q,R_cP,R_cU)$ we may therefore assume w.l.o.g. that $P(1)=1$. 

 When $n=m$, we can compose the isotopy $(Q,P,U)$ with a suitable power of the autotopism $(F_{\sigma^{-1},1},L_t,L_t)$ of Lemma \ref{le:shifty} of $A_1$. This replaces $P$ by $PL_t^e$, whose first column is the $(e+1)$-st column of $P$, and $Q$ by $QF_{\sigma^{-e},1}$. Since $F_{\sigma^{-e},1}(1)=1$, this does not change $Q(1)$, and we may assume $k_1=1$.  Then we can compose this autotopism with $({\rm id},R_c,R_c)$ to get  $P(1)=1$.

Note also that in this step we need the assumptions that $\sigma(a_i)\not=a_i$ and in case $n=m$ that $a$ lies in no proper subfield of $K$, otherwise not all these $u_i$ might be zero. It is not needed that the base field is finite.
In particular, if
$$P=\diag(R(h_{1,1})V,R(h_{2,2})V,\cdots, R(h_{m,m})V)$$
 for some $h_{i,i}\in K^\times$, call them $h_{i}=h_{i,i}$, then along the lines of Sandler's proof,
 $$h_{i+1}=h_1\,N_i^{\sigma}(h_2h_1^{-1}).$$
$(ii)$
  We obtain hence analogously as in \cite[p.~248]{Sandler1962} that the subgroup of autotopisms whose $P$ possesses the form
 $$P=\diag(R(h_1)V,R(h_2)V,\cdots, R(h_m)V)$$  is a normal subgroup of the autotopism group.
\end{proof}

\begin{theorem}\label{thm:important}
 Let
$(K/F,\sigma,a_1)=K[t;\sigma]/K[t;\sigma](t^m-a_1)$ and
 $(K/F,\sigma,a_2)=K[t;\sigma]/K[t;\sigma](t^m-a_2)$ be two proper nonassociative cyclic algebras of degree $m$ such that $a_i$ lies in no proper subfield of $K$ (i.e., $K=F(a_i)$).
If $(K/F,\sigma,a_1)=K[t;\sigma]/K[t;\sigma](t^m-a_1)$ and
 $(K/F,\sigma,a_2)=K[t;\sigma]/K[t;\sigma](t^m-a_2)$  are isotopic as algebras over some subfield $F_0$ of $F$, then they are isomorphic as algebras over $F_0$.
\end{theorem}

\begin{proof}
The proof is analogous to the proof of \cite[Theorem 9]{Sandler1962}, finiteness of fields is not used:
 Suppose that  the algebras are isotopic. Let $(Q,P,U)$ be such an isotopy. Recall that $k_1$ is the row of the unique nonzero entry in the first column of $P$.
By Theorem 4.3, $P=(H_{i,j})$ has exactly one nonzero entry in each column $j$, in row $k_j$. Column $1$ describes $P$ on the constants $y=y_1\in K$, so
$$P(y_1)=t^{k_1-1}h_{k_1,1}V(y_1),\quad\text{in particular}\quad P(1)=t^{k_1-1}h_{k_1,1}.$$
By Theorem 4.3 and  the proof of Theorem 4.4, we may assume w.l.o.g. that $P$ is diagonal with $P(1)=1$, so $P={\rm diag}(V,R(h_2)V,\dots,R(h_m)V)$ with $h_{k+1}=N_k^\sigma(h_2)$.

When  $m\geq 3$ we get
$$ V(a_1)=a_2  N_m^{\sigma}(h) $$
for some $h=h_2\in K^\times$.
This implies that $G_{V,h}$ is an isomorphism between the two rings by Theorem \ref{T:classifyisomorphism}.

When $n=m=2$ we obtain $$\sigma(V(a_1))=a_2 N_m^{\sigma}(h). $$
 This implies that $G_{\sigma V,h}$ is an isomorphism between the two rings (again by Theorem \ref{T:classifyisomorphism}).
 
 Since isotopy of the opposite algebras is equivalent to isotopy of the Petit algebras we have proved the assertion.
\end{proof}

The above proofs do not require that the field $F$ is finite.

To our knowledge, the following result is new also for semifields, as Sandler only proved it when $n=m$.

\begin{theorem}\label{thm:mnotn} Let $m\geq 3$ and let $A_1=(K[t;\sigma]/K[t;\sigma](t^m-a_1))^{op}$ and
$A_2=(K[t;\sigma]/K[t;\sigma](t^m-a_2))^{op}$ be  two proper nonassociative division algebras such that
 $n > m$ and  $\sigma(a_i)\not=a_i$. 
 If $A_1$ and $A_2$ are isotopic, then $A_1$ and $A_2$ are isomorphic.
 \end{theorem}

\begin{proof}
Let $(Q,P,U)$ be an isotopy between  $A_1$ and
$A_2$, that is $U(x\circ_1 y)=Q(x)\circ_2 P(y)$ for all $x,y\in A_1$. Then there exists some $V\in {\rm Aut}(K)$ such that $P=(H_{i,j})_{i,j}$ as in Theorem \ref{thm:isot} where the nonzero entries of $P$ are at $(k_j,j)$ with $1\le k_j\le m$ and $k_j\equiv k_1+(j-1)\,{ \rm mod}\, n$. Put $c=k_1-1\in\{0,\dots,m-1\}$, so $k_j-1\equiv c+(j-1) \,{ \rm mod}\,  n$ with $k_j-1\in\{0,\dots,m-1\}$.

Suppose $c\geq 1$ then $j=m-c+1\in \{2,\dots,m\}$ and $c+(j-1)=m$, so $k_j-1\equiv m \,{ \rm mod}\,  n$. As $m<n$, the only representative of $m$ in $\{0,\dots,n-1\}$ is $m$ itself, which is not in $\{0,\dots,m-1\}$. This is a contradiction.
Hence $c=0$, i.e. $k_1=1$ and $k_j=j$ for all $j$, so $P$ is diagonal.

Let $P(1)=h_1\in K^\times$.
For $c\in K={\rm Nuc}_r(A_2)$,  $({\rm id},R_c,R_c)$ is an autotopism of $A_2$, because $(x\circ_2y)\circ_2c=x\circ_2(y\circ_2c)$. This induces the isotopy $(Q,R_cP,R_cU)$. 
For $c=h_1^{-1}$ we get $R_cP(1)=1$.
By replacing $(Q,P,U)$ with $(Q,R_cP,R_cU)$ we may therefore assume w.l.o.g. that $P(1)=1$.
Then $Q=U=L_{Q(1)}P$, and arguing as in \cite[p.~251 ff., (60'')]{Sandler1962} we obtain $Q(1)=u_0\in K^\times$, $h_{i+1}=N_i^\sigma(h)$ with $h=h_2$,  and
$$V(a_1)=a_2\frac{\sigma^m(u_0)}{u_0}N_m^{\sigma}(h)=a_2\,N_m^{\sigma}(h\sigma(u_0)u_0^{-1}).$$
It follows that  $G_{V,h\sigma(u_0)u_0^{-1}}:K[t;\sigma]/K[t;\sigma](t^m-a_1)\to K[t;\sigma]/K[t;\sigma](t^m-a_2)$ is an isomorphism which induces an isomorphism $A_1\cong A_2$.
\end{proof}

The next example shows that non-isomorphic Sandler algebras constructed from distinct generators may be isotopic.

\begin{example}\label{ex:rel}
Let $F=\mathbb{F}_2$, $K=\mathbb{F}_8=\mathbb{F}_2(\alpha)$ with $\alpha^3=\alpha+1$,  and let
$\sigma(x)=x^2$.
We have
$$\alpha^7=1, \quad\alpha^4=\alpha^2+\alpha,\quad \alpha^5=\alpha^2+\alpha+1,\quad \alpha^6=\alpha^{-1}=\alpha^2+1.$$
Let $(K/F,\sigma^{-1},\alpha^6)$ and
$(K/F,\sigma,\alpha)$ be two proper nonassociative cyclic algebras of degree 3. We know that these are non-isomorphic semifields.
Let $A_1=K[t;\sigma^{-1}]/K[t;\sigma^{-1}](t^3-\alpha^6)^{op}$ and
$A_2=K[t;\sigma]/K[t;\sigma](t^3-\alpha)^{op}$.
Since $\alpha^6=\alpha^{-1}$, we have $a_1=a_2^{-1}$ here.
The semifields $A_1$ and $A_2$ are  isotopic: we define $\mathbb{F}_2$-linear bijective maps $P: A_1\to A_2$
$$P(x_0+tx_1+t^2x_2)=x_0+t\,\alpha^4x_2+t^2x_1,$$
and $Q=U=L_{t^2}P: A_1\to A_2$, so that
$$U(x_0+tx_1+t^2x_2)=\alpha^5x_2+t\,\alpha x_1+t^2x_0,$$
since $t^2\circ_2(z_0+tz_1+t^2z_2)=\alpha z_1+t\,\alpha z_2+t^2z_0$.
Note that $P(t)=t^2$ and that $U(1)=t^2\notin K$ is not of the form $G_{\tau,\alpha}$.

To check that $U(x\circ_1y)=U(x)\circ_2P(y)$ for all $x,y\in A_1$, it
suffices to check this for all $x=t^ib$, $y=t^jc$ with $b,c\in K$. We use that $b^8=b$ and that
$(t^ib)\circ_1(t^jc)=t^it^jb^{4^j}c$ with $t^3=\alpha^6$, and $(t^ib)\circ_2(t^jc)=t^it^jb^{2^j}c$ with $t^3=\alpha$.
 We compute
\\ $U(bc)=t^2bc=(t^2b)\circ_2c$,
\\
$U(b\circ_1 tc)=U(tb^4c)=t\alpha b^4c=(t^2b)\circ_2(t^2c)=t^4b^4c$,
 \\
$U(b\circ_1 t^2c)=U(t^2b^2c)=\alpha^5b^2c=(t^2b)\circ_2(t\alpha^4c)$,
\\
$U(tb\circ_1 c)=U(tbc)=t\alpha bc=(t\alpha b)\circ_2c$,
\\  $U(tb\circ_1 tc)=U(t^2b^4c)=\alpha^5b^4c=(t\alpha b)\circ_2(t^2c)$,
\\ $U(tb\circ_1 t^2c)=U(\alpha^6b^2c)=t^2\alpha^6b^2c= t^2\alpha^2b^2\alpha^4c= (t\alpha b)\circ_2(t\alpha^4c)$,
\\  $U(t^2b\circ_1 c)=U(t^2bc)=\alpha^5bc=(\alpha^5b)\circ_2c$,
\\ $U(t^2b\circ_1 tc)=U(\alpha^6b^4c)=t^2\alpha^6b^4c=(\alpha^5b)\circ_2(t^2c)=t^2\alpha^{20}b^4c$,
\\  $U(t^2b\circ_1 t^2c)=U(t\alpha^6b^2c)=t\alpha^7b^2c=tb^2c=t\alpha^{10}b^2\alpha^4c=
 (\alpha^5b)\circ_2(t\alpha^4c)$.
\end{example}

When we use any skew polynomial $f\in K[t;\sigma]$ other than $t^m-a$ and two generators $\sigma_i$, isotopy and isomorphism classes need not be the same; isotopy remains a much coarser classification.

\begin{example}\label{ex:rel2}
Let us return to the setup of Example \ref{ex:rel}.
Let $F=\mathbb{F}_2$, $K=\mathbb{F}_8=\mathbb{F}_2(\alpha)$ with $\alpha^3=\alpha+1$,  and let
$\sigma(x)=x^2$. Let $f(t)$ be monic and irreducible of degree 3. We will show that there are at least 7 isomorphism classes and exactly two isotopy classes of Petit algebras $K[t;\sigma]/K[t;\sigma]f$.

We use the notation from \cite{Pum2025}. Over finite fields the bound $f^*$ of a monic irreducible skew polynomial $f$ has maximal degree $mn=9$. So $f^*=\hat h(t^3)$ with $\hat h\in F[x]$ monic and irreducible of degree $m_0=3$, i.e. $\hat h\in I(\mathbb{F}_2,3)$.
Now there are exactly two irreducible cubics in $\mathbb{F}_2[x]$, so that
$I(\mathbb{F}_2,3)=\{x^3+x+1,\ x^3+x^2+1\}.$ For each $\hat h\in I(\mathbb{F}_2,3)$, the polynomial $\hat h(t^3)$ has degree $9$ and factors into $k=3$ irreducible factors of degree $3$ in $R$. Any one of these factors is an $f$ with bound $\hat h(t^3)$. Here, $G=\mathbb{F}_2^\times\rtimes {\rm Gal}(\mathbb{F}_8/\mathbb{F}_2)$. Since $\lambda\in\mathbb{F}_2^\times=\{1\}$ and every $\tau$ fixes the coefficients of $\hat h$, which lie in $\mathbb{F}_2$, the action $\hat h\mapsto\lambda^{-3}\hat h^\tau(\lambda x)$ is trivial. So there are exactly two $G$-orbits, i.e. exactly two isotopy classes \cite[Corollary 5.12]{SantonastasoZhou2026}.

For all $\tau\in {\rm Aut}(K)_\sigma$, $\alpha\in K^\times$, and for any monic $f\in R$ define the monic polynomial
$$f^{(\alpha,\tau)}(t)=N^\sigma_m(\alpha^{-1}) f^\tau(\alpha t) =N^\sigma_m(\alpha^{-1})\sum_{i=0}^m \tau(a_i) N^\sigma_i(\alpha) t^i$$
 for all $\tau\in \Aut(K)_\sigma$ and $\alpha\in K^\times$.

When $n\geq m$ then 
every monic $g$ with $ K[t;\sigma]/K[t;\sigma]f\cong K[t;\sigma]/K[t;\sigma]g$ is of the form $g(t)=f^{(\alpha,\tau)}(t)$. This means the isomorphism class of $f$ has at most $|{\rm Aut}(K)_\sigma||K^\times|=3\cdot 7=21$ elements (different pairs $(\alpha,\tau)$ may give the same $g$). 

By \cite[Theorem 3.2]{Pum2025} we know that $K[t;\sigma]/K[t;\sigma]h\cong M_3(\mathbb{F}_8)$.
  An irreducible $f$ has bound $h(t)=\hat h(t^3)$ exactly when it right-divides $h$ \cite[Lemma 3.1]{Pum2025}. So these $f$ correspond bijectively to the maximal left ideals of $M_3(\mathbb{F}_8)$.
 The maximal left ideals of $M_3(\mathbb{F}_8)$ are of the form $M_v=\{X \in M_3(\mathbb{F}_8) \,|\, Xv=0\}$ for some nonzero $v\in\mathbb{F}_8^3$. Here, $M_v=M_{v'}$ if and only if  $v=cv'$ for some $c\in \mathbb{F}_8^\times$. Thus there are exactly
$\frac{8^3-1}{8-1}=73$ maximal left ideals. Every irreducible $f$ of degree 3 has one of the two bounds above, and distinct bounds give disjoint sets of $f$, we thus obtain
$\ell=2\cdot 73=146$
 monic irreducible $f\in K[t;\sigma]$ of degree $3$,  therefore there will be at least $\lceil \ell/21\rceil=\lceil146/21\rceil=7$ distinct isomorphism classes.
\end{example}

 In Section \ref{sec:counting}, we will see that our results on Sandler algebras  enable us to count the isotopism classes of general Petit algebras $K[t,\sigma]/K[t,\sigma]f$ when the degree $m$ of the monic irreducible skew polynomial $f$ is the same as the order $m$ of $\sigma$.

\section{Counting  isomorphic Sandler semifields}\label{sec:counting}

Let $\os,\de\in \mathbb{N}$, $p$ be a prime and $q=p^\gamma$.  Set  $F_0=\mathbb{F}_p$, $F=\mathbb{F}_q$ and $K=\mathbb{F}_{q^\os}$. Then $\Gal(K/F)$ is cyclic of order $\os$, generated by $\sigma$ where $\sigma(x)=x^q$ for all $x\in K$.  Let $K^\times =\langle \xi \rangle$, where $\xi$ is a primitive element of $K$, that is, of order $q^\os-1$.   For any $x\in K^\times$ we write
$$
N^\sigma_\de(x) = \prod_{i=0}^{m-1}\sigma^i(x).
$$
Define
$$
s = \gcd(q^\os-1, \frac{q^\de-1}{q-1}).
$$
Note that when $\gcd(\de,\os)>1$ we have $s>1$, since as polynomials we have $\gcd(x^\os-1,x^\de-1)=x^{\gcd(\os,\de)}-1$. In any case, $\gcd(s,p)=1$ since  $s\mid q^n-1$.

For  irreducible skew polynomials $t^m-a\in K[t;\sigma]$ that do not generate two-sided ideals,
the rings $K[t;\sigma]/K[t;\sigma] (t^m-a)$ which are division algebras are  the duals (i.e., the opposite algebras) of Sandler's semifields.

 We first note that there are no Sandler semifields of order $q^{mn}$ with left and middle nucleus $K$  when $s=1$.

\begin{lemma}
Let $s=1$. Then there are no Sandler semifields of order $q^{mn}$ that are isomorphic to the opposite algebra of  $K[t;\sigma]/K[t;\sigma](t^\de-a)$.
\end{lemma}

\begin{proof}
If $s=1$ then  every $a\in K^\times$ is of the form $a=N^{\sigma}_\de(\alpha)$ for some $\alpha \in K^\times$ which is equivalent to $(t-\alpha)$ right divides $t^\de-a$, so that $t^\de-a$ is a reducible polynomial.
 If there exists a Sandler semifield with the given parameters then it is isomorphic to the opposite algebra of  $K[t;\sigma]/K[t;\sigma](t^\de-a)$ which is not a division algebra, however, a contradiction.
\end{proof}

When studying and counting semifields, we need to consider isomorphisms and isotopies of the algebras $K[t;\sigma]/K[t;\sigma](t^m-a)$ over their prime field $F_0=\mathbb{F}_p$, in other words, we look at the $K[t;\sigma]/K[t;\sigma](t^m-a)$ now as proper nonassociative division rings.
The Petit ring   $K[t;\sigma]/K[t;\sigma](t^m-a)$ is said to be  \emph{proper nonassociative} when it is not an associative ring; by \cite[(15)]{P66} this occurs exactly when $\sigma(a)\neq a$ or $\os\nmid \de$.

For all semifields
$K[t;\sigma]/K[t;\sigma](t^m-a)$, every monomial isomorphism between them must have degree one
\cite[Corollary 4.5]{NevinsPumpluen2026codes}.
Indeed, when $n\nmid m$ or $n=m$ and $a\in K\setminus F$, then  \emph{every} isomorphism is an isometry $G_{\tau,\alpha}=G_{\tau,\alpha,1}$, that is an equivalence \cite{NevinsPumpluen2025submitted}.

Indeed, in a large number of cases \emph{every} isomorphism between proper nonassociative rings of the form $K[t;\sigma]/K[t;\sigma](t^m-a)$
 is of the form $G_{\tau,\alpha}$; this is the case when  $\os \nmid\de$ or alternatively,  when \eqref{stareq} holds \cite[Corollary 6.2, Theorem 6.3]{NevinsPumpluen2026codes}.

\begin{theorem}\label{thm:main}
    Suppose $G: K[t;\sigma]/K[t;\sigma](t^m-a)\to K[t;\sigma]/K[t;\sigma](t^m-b) $ is a non monomial homomorphism satisfying $G|_{K}=\tau$ for some $\tau\in \Aut(K)$. Let $k=\deg(G(t))$.  Suppose further that
    \begin{equation}\label{stareq}
        \tag{$\star$} \text{if $u_0>1$ is the least value for which $u_0k\geq \de$, then $(u_0+1)k< 2\de$}.
    \end{equation}
    Then $K[t;\sigma]/K[t;\sigma](t^m-b)$
    must be associative, that is, $b\in F^\times$ and $\os\mid\de$.
\end{theorem}

This follows from \cite[Theorem 6.3]{NevinsPumpluen2026codes}. The hypothesis \eqref{stareq} is a minor technical condition that should not be necessary.

\begin{conjecture}
    \label{conj:nonmonomial}
    Suppose $G:  K[t;\sigma]/K[t;\sigma](t^m-a)\to K[t;\sigma]/K[t;\sigma] (t^m-b)$ is a non monomial isomorphism satisfying $G|_{\mathbb{F}_{q^n}}=\tau$ for some $\tau\in \Aut(K)$.  Then  $a,b\in F$ and $\os\mid \de$, that is, $K[t;\sigma]/K[t;\sigma] (t^m-b)$ must be associative.
\end{conjecture}

We present work towards proving the conjecture in \cite{NevinsPumpluen2026unpublished}.

The following results will therefore only look at the isomorphisms $G_{\tau,\alpha}=G_{\tau,\alpha,1}$, anticipating that these are the only existing ones not just when  $\os \nmid\de$.

For each divisor $d$ of $s$, write $\varphi(d)=|(\Z/d\Z)^\times|$ for the Euler phi function of $d$ and for any $\ell$ such that $\gcd(\ell,d)=1$, write $o(\ell)_d$ for the order of the image of $\ell$ in $(\Z/d\Z)^\times$.

\begin{theorem}\label{thm:N}
(i)  The number of equivalence classes of  Petit rings $K[t;\sigma]/K[t;\sigma](t^m-a)$  as $a$ varies over $K^\times$,
  is
    $$
   N= \sum_{d|s}\frac{\varphi(d)}{o(p)_d}.
    $$
    (ii) When $\os \nmid\de$ this is also the number of \emph{isomorphism} classes  of these rings.
\end{theorem}

\begin{proof}
$(i)$ is \cite[Theorem 3.7]{NevinsPumpluen2026codes}\label{thm:3.7}.
\\ $(ii)$
 When $n\nmid m$, every isomorphism between two such rings has the form
$G_{\tau,\alpha}$ by Theorem \ref{T:classifyisomorphism}$(i)$. 
 Hence $N$ is also the number of isomorphism classes  of these rings.
\end{proof}

Here, algebras in one equivalence class are all isomorphic to each other via  isomorphisms $G_{\tau,\alpha}$, and  algebras in one isometry class are all isomorphic to each other via some isomorphism $G_{\tau,\alpha,k}$. Algebras need not be division algebras in this count nor do they need to be proper nonassociative.

Computing the order of $p$ modulo various divisors $d$ of $s$ is an arithmetically interesting problem.  It is simple only in some cases; for example,  $o(p)_d=1$ for all $d|\gcd(s,p-1)$ since in that case $p\equiv 1\mod d$.

 Let us first consider the case that we have $\os\mid \de$. Then $s>1$ and we always have some equivalence classes with associative and some with proper nonassociative rings  of the form  $K[t;\sigma]/K[t;\sigma](t^\de-a)$.

\begin{theorem}\label{thm:nonassociativecase}
    Suppose $\os\mid \de$, let $s = \gcd(q^\os-1, \frac{q^\de-1}{q-1})$ and set $s_0 = \gcd(q-1,\frac{\de}{\os})$. Fix $\sigma$. Then
    the number of isometry classes of  Petit rings $\Kone{a}$ that are \emph{proper nonassociative} rings as $a$ runs over $K^\times$, is given by
    $$
     \sum_{\substack{d \mid s \\ d \nmid s_0}} \frac{\varphi(d)}{o(p)_d}.
    $$
    Moreover, when $\os =\de$ this is also the number of isomorphism classes  of these rings.
\end{theorem}

\begin{proof}
First note that $s_0$ is a divisor of $s$ since
 $s = (\frac{q^\os-1}{q-1}) s_0$
by  \cite[Lemma 3.9]{NevinsPumpluen2026codes}. By  \cite[Corollary 3.10]{NevinsPumpluen2026codes}, the number of equivalence classes of Petit rings $\Kone{a}$ that are associative, that is, where $a$ runs over $F^\times$, is given by
$$
\sum_{d|s_0}\frac{\varphi(d)}{o(p)_d}.
$$
 By Theorem \ref{thm:N}, the number of equivalence classes of
  Petit algebras $K[t;\sigma_1]/K[t;\sigma_1](t^m-a)$  as $a$ varies over $K^\times$, is
  $$
 \sum_{d|s}\frac{\varphi(d)}{o(p)_d}.
    $$
 The number of equivalence classes of proper nonassociative Petit rings thus is
     $$
   \sum_{d|s}\frac{\varphi(d)}{o(p)_d}-\sum_{d|s_0}\frac{\varphi(d)}{o(p)_d}=  \sum_{\substack{d \mid s \\ d \nmid s_0}} \frac{\varphi(d)}{o(p)_d}.
    $$
For proper nonassociative Petit algebras, every isometry is indeed an equivalence, so that this is also exactly the number of isometry classes.
\end{proof}

\begin{corollary}\label{cor:nonassociativecase}
    Suppose $\os\mid \de$, $s = \gcd(q^\os-1, \frac{q^\de-1}{q-1})$ and set $s_0 = \gcd(q-1,\frac{\de}{\os})$. Fix $\sigma$. Then
    the number of isometry classes of semifields $K[t;\sigma]/K[t;\sigma](t^m-a)$  as $a$ runs over $K^\times$, is bounded above by
    $$
    M= \sum_{\substack{d \mid s \\ d \nmid s_0}} \frac{\varphi(d)}{o(p)_d}.
    $$
    In particular, when $\os =\de$ then the number of isomorphism classes of Sandler semifields $(\Kone{a})^{op}$  as $a$ runs over $K^\times$, is bounded above by $M$.
\end{corollary}

Indeed, when $n=m$ we have more.

\begin{theorem}\label{thm:cyclicmain}
Let $n=m$ and $s = \gcd(q^\os-1, \frac{q^\de-1}{q-1})$. Let $e$ be a proper divisor of $n$ and put $s_e=\frac{q^e-1}{q-1}$. Then the number of isomorphism (and isotopy) classes of semifields
$(K/F,\sigma,a)$ for a fixed choice of $\sigma$ is exactly
\[
 M_0 =\sum_{\substack{d\mid s\\ d\nmid s_e\ \forall e\mid n,\ e<n}}
  \frac{\varphi(d)}{o(p)_d}.
\]
Varying the choice of $\sigma$ yields a total of $\varphi(n) M_0$  isomorphism classes, in $\varphi(n)$ families.
\end{theorem}

\begin{proof}
 If $e$ is a proper divisor of $n$, then $\mathbb{F}_{q^e}$ is a proper subfield of $K/F$. Now $a\in\mathbb{F}_{q^e}$ if and only if the class of $a$ lies in the
subgroup of order $s_e$ of $\mathbb{Z}/s \mathbb{Z}$. When counting non-isomorphic semifields for $n=m$, we must exclude all those $a\in K^\times$ which lie in a proper subfield of $K/F$ which is why we sum over all $d$ with $d\nmid s_e$ here.
 Since moreover here $N_n^\sigma(K^\times)=F^\times$ and
the condition $F(a)=K$ is invariant under $F^\times$ and under $\Aut(K)$, the
number of isomorphism classes of semifields
$(\mathbb{F}_{q^n}/\mathbb{F}_q,\sigma,a)$ is as asserted. By \cite[Theorem 9]{Sandler1962} this is then also the number of
isotopy classes for a fixed $\sigma$.
\end{proof}

\begin{corollary}\label{cor:nonassociativecase2}
Let  $m=n$ be prime.  Fix $\sigma$. Let $e$ be a proper divisor of $n$ and put $s_e=\frac{q^e-1}{q-1}$. Then all proper nonassociative cyclic algebras $\Kone{a}$ are division rings, and the number of isomorphism  classes of semifields  $K[t;\sigma]/K[t;\sigma](t^m-a)$   as $a$ runs over $K^\times$ is exactly
      $$
    M_0= \sum_{d \mid s, \, d \nmid s_0} \frac{\varphi(d)}{o(p)_d}=\sum_{d\mid s,\,d>1}\frac{\varphi(d)}{o(p)_d}.
    $$
    This is also the number of isotopy classes.

\end{corollary}

\begin{proof}
This is Theorem \ref{thm:cyclicmain} for $m$ prime  since in this case $e=1$, and $s_1=1$.  By Theorem \ref{thm:important}, isotopic is the same as isomorphic.
\end{proof}

\begin{lemma}\label{le:boundbinomial}
Let $n=m$ and let $a\in K$ with $F(a)=K$. Let $m_a(x)\in F[x]$ be the minimal
polynomial of $a$ over $F$. Then $m_a(t^n)$ is the bound of $t^n-a\in K[t;\sigma]$.
\end{lemma}

\begin{proof}
Since $F(a)=K$, the conjugates
$a,\sigma(a),\dots,\sigma^{n-1}(a)$ are pairwise distinct, and
$m_a(x)=\prod_{i=0}^{n-1}(x-\sigma^i(a))=g(x)(x-a)$ for some $g\in K[x]$, that means $m_a(t^n)=g(t^n)(t^n-a)$.  Thus the minimal central left
multiple $\hat{h}(t^n)$ of $t^n-a$ must divide $m_a(t^n)$. That means that $\hat{h}(x)$ divides $m_a(x)$ in
$F[x]$. Since $m_a$ is irreducible and both are monic, this implies that $\hat h=m_a$.
\end{proof}

 We now turn to the case that $\os\nmid \de$. We  get an upper bound as a direct corollary from Theorem \ref{thm:N}$(ii)$, recalling that when  $n \nmid m$,  the class corresponding to $d=1$, $a\in N^\sigma_m(K^\times)$, is never a division algebra, since $t-z$ right-divides $t^m-a$, and that all  isomorphisms between proper nonassociative algebras of the form $K[t;\sigma]/K[t;\sigma](t^m-a)$ are monomial of degree one.

\begin{corollary}\label{cor:1}
Let $\os \nmid\de$. Fix $\sigma$. Then the number of isomorphism classes of   semifields  $K[t;\sigma]/K[t;\sigma](t^m-a)$  as $a$ varies over $K^\times$ is bounded above by $N-1$ with
    $$
   N= \sum_{d|s}\frac{\varphi(d)}{o(p)_d}.
    $$
\end{corollary}

When $m$ is prime we know a lot more.

\begin{lemma}\label{le:main}
Let  $m=2$ or $3$, or let $m\geq 5$ be prime and suppose that $F$ contains a primitive $m$th root of unity, that is $\mu_m\subset F$. Then
 \\ (i) $K[t;\sigma]/K[t;\sigma](t^\de-a)$ is a division ring if and only if it is not  isometric
  to $K[t;\sigma]/K[t;\sigma](t^\de-1)$,
  \\ (ii) If $K[t;\sigma]/K[t;\sigma](t^\de-a)$ is a division ring then it is not isomorphic
  to $K[t;\sigma]/K[t;\sigma](t^\de-1)$.
\end{lemma}

\begin{proof}
$(i)$ For every $\tau\in \Aut(K)$, there is an isomorphism $$G_{\tau,\alpha}:K[t;\sigma]/K[t;\sigma](t^m-a)\to K[t;\sigma]/K[t;\sigma](t^m-1)$$
 if and only if there is $\alpha\in K^\times$ such that $\tau(a)=N_m^\sigma(\alpha)$. We know that $\tau(a)=N_m^\sigma(\alpha)$ is equivalent to $a=N_m^\sigma(\tau^{-1}(\alpha))$, since $\tau^{-1}(N_m^\sigma(\alpha))=N_m^\sigma(\tau^{-1}(\alpha))$.
 In other words, $K[t;\sigma]/K[t;\sigma](t^m-a)$ is  isomorphic to $K[t;\sigma]/K[t;\sigma](t^m-1)$ via  $G_{\tau,\alpha}$ if and only if $N_m^\sigma(\alpha)=a$
for some $\alpha\in K^\times$.  By \cite[Theorem 3.11]{CB}, however, $K[t;\sigma]/K[t;\sigma](t^m-a) $ is a division ring
 if and only if
$N_m^\sigma(\alpha)\not=a$ for all $\alpha\in K^\times$.
\\ $(ii)$ Since $t-1$ right-divides $t^m-1$, we know that $K[t;\sigma]/(t^m-1)$ is not a division algebra. 
\end{proof}

\begin{corollary}\label{cor:nonassociativecase2}
Let  $m=2$ or $3$, or let $m\geq 5$ be prime and suppose that $F$ contains a primitive $m$th root of unity, that is $\mu_m\subset F$. Fix $\sigma$.
Let  $n\nmid m$. Then the number of isomorphism  classes of the semifields  $\Kone{a}$ as $a$ runs over $K^\times$ is exactly $N-1$.  When $n>m\geq 3$, then the number of isotopy classes of the semifields $\Kone{a}$ with $a\in K\setminus F$ equals the number of their isomorphism classes.
   \end{corollary}

\begin{proof}
By Theorem \ref{thm:N}$(ii)$, as $a$ runs over $K^\times$, the number of isomorphism  classes of the Petit rings  $\Kone{a}$ is exactly
$N= \sum_{d|s}\frac{\varphi(d)}{o(p)_d} $. This initial count includes rings which are not division rings.
 However, the summand for $d=1$ in the sum is the class of rings isomorphic to $\Kone{1}$, which is not a division ring by Lemma \ref{le:main}.
 By Lemma \ref{le:main}, every other proper nonassociative algebra is a division ring. This implies the assertion.
 Theorem \ref{thm:mnotn} yields the second assertion.
  \end{proof}

Note that when  $\os=\de$  then all nonassociative rings that are isomorphic  are equivalent \cite{NevinsPumpluen2025submitted}. Moreover, for Sandler semifields with $n\ge m$ and fixed $\sigma$, where $m\geq 3$ when $n>m$,
 isomorphisms and isotopy are the same; Sandler only proved this for $n=m$ \cite[Theorem 9]{Sandler1962}.

\section{Counting the isotopy classes of all Petit semifields}\label{sec:new}

We first show that $M_0$ from Theorem \ref{thm:cyclicmain} is the exact number of isotopy classes of all Petit semifields $K[t;\sigma]/K[t;\sigma]f$ when $n=m\geq 2$ and $f$ is monic and irreducible  of degree $m$. Most surprisingly perhaps, in this setting  every semifield $K[t;\sigma]/K[t;\sigma]f$, with $f\in K[t;\sigma]$ irreducible of
degree $n=m$, is isotopic to a suitable proper nonassociative cyclic algebra.

\begin{theorem}\label{thm:allcyclicmain}
Let $n=m\geq 2$ and $s = \gcd(q^\os-1, \frac{q^\de-1}{q-1})$. Let $e$ be a proper divisor of $n$ and put $s_e=\frac{q^e-1}{q-1}$. Fix a generator $\sigma$ of ${\rm Gal}(\mathbb{F}_{q^n}/\mathbb{F}_q)$.
Then every semifield $K[t;\sigma]/K[t;\sigma]f$, with $f\in K[t;\sigma]$ irreducible of
degree $n$, is isotopic to a proper nonassociative cyclic algebra $(K/F,\sigma,a)$ with $F(a)=K$.
In particular, the number of isotopy classes of the semifields
$K[t;\sigma]/K[t;\sigma]f$, $f$ irreducible of degree $n$, is exactly
$$M_0=\sum_{\substack{d\mid s\\ d\nmid s_e\ \forall e\mid n,\, e<n}}\frac{\varphi(d)}{o(p)_d}.$$
\end{theorem}

\begin{proof}
Let $f$ be irreducible of degree $n$. Over finite fields the bound $f^*$ of $f$ has maximal
degree $n^2$, so $f^*(t)=\hat h(t^n)$ with $\hat h(x)\in F[x]$ monic irreducible of degree $n$
(\cite[Theorem 3.2]{Pum2025}). Since $\hat h(x)$ has degree $n$, it must have a zero over
$K=\mathbb{F}_{q^n}$. Let $a\in K$ be such a zero, then we get $F(a)=K$ and hence it follows that $\hat h=m_a$. By
Lemma \ref{le:boundbinomial}, $t^n-a$ has the same bound as $f$, hence $f$ and $t^n-a$
are similar (\cite[Theorem 3.4]{Pum2025}). Therefore $K[t;\sigma]/K[t;\sigma]f$ and
$(K/F,\sigma,a)$ are isotopic (\cite[Theorem 3.6]{Pum2025}).

Thus every isotopy class contains some $(K/F,\sigma,a)$ with $F(a)=K$. By Theorem
\ref{thm:important}, two such algebras are isotopic if and only if
they are isomorphic. By Theorem \ref{thm:cyclicmain}, there are exactly $M_0$ isomorphism classes of them.
\end{proof}

 Let us compare the above with \cite{SantonastasoZhou2026}. We use their terminology.
For $n>2$, and fixed $\sigma$, the number of isotopy classes of Petit
semifields of order $q^{nm}$ with left and middle nucleus $\mathbb{F}_{q^n}$ and right
nucleus $\mathbb{F}_{q^m}$ equals the number of $\widehat{\mathcal G}$-orbits on the set
$\mathcal{I}_{q,m}$ of monic irreducible polynomials of degree $m$ over $\mathbb{F}_q$
\cite[Corollary 5.12]{SantonastasoZhou2026}, where an orbit count is presented and a lower bound for this number is given.
For $m=n$, Theorem \ref{thm:allcyclicmain} shows that the number of isotopy classes  is exactly $M_0$. Thus $M_0$ must also be the number of
$\widehat{\mathcal G}$-orbits on
$\mathcal I_{q,m}$ when $m=n$.

This can also be seen directly: by Lemma \ref{le:boundbinomial}, $a\mapsto m_a$ maps the
elements $a\in K$ with $F(a)=K$ onto $\mathcal{I}_{q,n}$. Under this map, for all $\lambda\in F^\times=N_{K/F}(K^\times)$ and all
$\tau\in{\rm Aut}(K)$, the action
$a\mapsto\lambda^{-1}\tau(a)$ which describes the isomorphisms between the
algebras $(K/F,\sigma,a)$ (Theorem \ref{T:classifyisomorphism}), corresponds to the action
$\hat h\mapsto\lambda^{-n}\hat h^\tau(\lambda x)$ of $\widehat{\mathcal G}$.

This observation makes us flip our point of view. The $\widehat{\mathcal G}$-orbits on the set
$\mathcal{I}_{q,m}$ of monic irreducible polynomials  of degree $m$ in $\mathbb{F}_q[x]$ (we use the notation of \cite{SantonastasoZhou2026}) coincide with the $G$-orbits on $I(q,m)=\mathcal{I}_{q,m}$  in the notation of  \cite{Pum2025}.

The number of these orbits is exactly the number of isotopy classes of Petit semifields $K[t;\sigma]/K[t;\sigma]f$ of order $q^{nm}$ obtained from \emph{all} irreducible skew polynomials $f\in \mathbb{F}_{q^n}[t;\sigma]$ of degree $m$ \cite[Corollary 5.12]{SantonastasoZhou2026}.

 We now re-compute the number $M(q,m)$ of isotopy classes of Petit semifields $K[t;\sigma]/K[t;\sigma]f$ of order $q^{nm}$ obtained from \emph{all} irreducible skew polynomials $f\in \mathbb{F}_{q^n}[t;\sigma]$ of degree $m$, which was obtained  in \cite[Theorem 4.3]{Pum2025}, for any choice of $n>2$ and $m$.

\begin{theorem} \label{thm:orbitcounting}
Let $n>2$, $m> 1$, and define $s_e=\frac{q^e-1}{q-1}$ for all $e\mid m$. Fix a generator $\sigma$ of ${\rm Gal}(\mathbb{F}_{q^n}/\mathbb{F}_q)$. Then the number of
isotopy classes of the semifields $\mathbb{F}_{q^n}[t;\sigma]/\mathbb{F}_{q^n}[t;\sigma]f$, with $f$ irreducible of
degree $m$, is exactly 
$$M(q,m)=\sum_{\substack{d\mid s_m\\ d\nmid s_e\ \forall\, e\mid m,\ e<m}}\frac{\varphi(d)}{o(p)_d}.$$
\end{theorem}

\begin{proof}
We use the notation from \cite[Section 4]{Pum2025}.
Let $G=\mathbb{F}_q^\times\rtimes{\rm Gal}(\mathbb{F}_q/\mathbb{F}_p)$. Then $G$ act on the set $I(q,m)$ of irreducible monic polynomials in $\mathbb{F}_q[x]$ of degree $m$ via
$h\mapsto\lambda^{-m}h^\tau(\lambda x)$.

Put $M=\mathbb{F}_{q^m}$ and $M^{(m)}=\{a\in M\mid \mathbb{F}_q(a)=M\}$. Each $h\in I(q,m)$ is the
minimal polynomial of exactly $m$ elements of $M^{(m)}$, which form one orbit under
${\rm Gal}(M/\mathbb{F}_q)$. Extend each $\tau \in {\rm Gal}(\mathbb{F}_q/\mathbb{F}_p)$
 to ${\rm Gal}(M/\mathbb{F}_p)$ by $\tau_r(u)=u^{p^r}$ for a suitable $r$. If $h(a)=0$, then $h^{(\lambda,\tau)}(\lambda^{-1}\tau(a))=0$. Hence the $G$-orbits on $I(q,m)$ correspond
bijectively to the orbits on $M^{(m)}$ of the group generated by multiplication by
$\mathbb{F}_q^\times$ and by the Frobenius automorphism $u\mapsto u^p$; note that this group contains
${\rm Gal}(M/\mathbb{F}_q)$.

We know that $M^\times$ is cyclic, and that we have $M^\times/\mathbb{F}_q^\times\cong\mathbb{Z}/s_m\mathbb{Z}$. Now the Frobenius automorphism
 acts on this quotient as multiplication by $p$, and for $e\mid m$, the image of
$\mathbb{F}_{q^e}^\times$ is the subgroup of order $s_e$. As $\mathbb{F}_q^\times\subset
\mathbb{F}_{q^e}^\times$, an element $a$ lies in $\mathbb{F}_{q^e}$ if and only if its class lies in this
subgroup. Hence $a\in M^{(m)}$ if and only if the order $d$ of its class does not divide
any $s_e$ with $e\mid m$, $e<m$. Finally, multiplication by $p$ permutes the
$\varphi(d)$ elements of order $d$ in $\mathbb{Z}/s_m\mathbb{Z}$ in orbits of length $o(p)_d$, which gives
$\varphi(d)/o(p)_d$ orbits for each such $d$. We obtain a total of exactly 
$$M(q,m)=\sum_{\substack{d\mid s_m\ d\nmid s_e \forall\, e\mid m, e<m}}\frac{\varphi(d)}{o(p)_d}$$
 distinct $G$-orbits on $I(q,m)$. By \cite[Corollary 5.12]{SantonastasoZhou2026}, this also equals the number of
$\widehat{\mathcal{G}}$-orbits on $I(q,m)$, which coincide with the $G$-orbits above, and is the number of isotopy classes.
 \end{proof}
 
 Note that for $m=n$ we have $s=s_n=\frac{q^n-1}{q-1}$, so that now Theorem \ref{thm:allcyclicmain} is a special case of Theorem \ref{thm:orbitcounting}
when $n=m>2$, but note that Theorem \ref{thm:allcyclicmain} also holds when $n-m-2$.

Two semifields coordinatize the same non-Desarguesian plane if and only if they are isotopic. We have  thus also
counted the corresponding non-Desarguesian planes explicitly  up to isotopy.

The value $M(q,m)$ can be arithmetically difficult to compute when $s_m$ is very large. We can bound it as follows, improving on the upper bound 
 $$\frac{q^m-\theta}{m\gamma(q-1)} \leq M(q,m) \leq \frac{q^m-\theta}{m}$$
presented in \cite{LS}, where $q=p^\gamma$ and $\theta$ is the number of elements of
$\mathbb{F}_{q^m}$ lying in a proper subfield containing $\mathbb{F}_q$.
Observe that $q^m-\theta=\sum_{d\mid m}\mu(d)q^{m/d}$ where $\mu$ is the Moebius function. 

\begin{proposition}\label{prop:bounding}
Let $q=p^\gamma$, $s=s_m=\frac{q^m-1}{q-1}$, and let $\theta$ be the number of elements of
$\mathbb{F}_{q^m}$ lying in a proper subfield containing $\mathbb{F}_q$.
\\ (i) We have
$$\frac{q^m-\theta}{m\gamma(q-1)} \leq M(q,m) \leq \frac{q^m-\theta}{m\gamma(q-1)}
+\frac{1}{m\gamma}\sum_{k=1}^{m\gamma-1}\gcd(p^k-1,s) < \frac{q^m-\theta}{m\gamma(q-1)}+q^{m/2}.$$
In particular, we have $$M(q,m)\sim\dfrac{q^{m-1}}{m\gamma}$$ as $q\to\infty.$
\\ (ii) $M(q,m)$ is greater or equal to the
number of divisors $d$ of $s$ with $d\nmid s_e$ for all $e\mid m$, $e<m$.
\end{proposition}

\begin{proof}
$(i)$ Let $S\subset\mathbb{Z}/s\mathbb{Z}$ be the set of elements 
of order $d$ in $\mathbb{Z}/s\mathbb{Z}$ with $d\nmid s_e$ for all proper
divisors $e$ of $m$. By the proof of Theorem \ref{thm:orbitcounting}, there is a one-on-one correspondence between $S$ and the set
$\{a\in\mathbb{F}_{q^m}\mid\mathbb{F}_q(a)=\mathbb{F}_{q^m}\}/\mathbb{F}_q^\times$, so 
$|S|=\frac{q^m-\theta}{q-1}$
 and $M(q,m)$ is the
number of orbits of multiplication by $p$ on $S$. Since $p^{m\gamma}=q^m\equiv 1 \,{\rm mod}\, s$, the group
$\mathbb{Z}/m\gamma \mathbb{Z}$ acts on $S$ via $k\cdot x=p^kx$, and by Burnside's Lemma we have
$$M(q,m)=\frac{1}{m\gamma}\sum_{k=0}^{m\gamma-1}|\{x\in S\mid p^k x=x\}| .$$
For $k=0$ we have the summand $|S|$, which gives us our lower bound. For all integers $k$ with $1\leq k\leq m\gamma-1$, the number of
 $x\in S$ with $p^k x=x$ is less or equal to the number of solutions of $(p^k-1)x=0$ in $\mathbb{Z}/s\mathbb{Z}$, which is
$\gcd(p^k-1,s)$. Finally, we have
$$\gcd(p^k-1,s)\leq\gcd(p^k-1,q^m-1)=p^{\gcd(k,m\gamma)}-1\leq p^{m\gamma/2}-1=q^{m/2}-1,$$
 since we have
$\gcd(k,m \gamma)\leq m\gamma/2$ for $1\leq k<m\gamma$. 
\\ $(ii)$ Each summand
$\varphi(d)/o(p)_d$ in Theorem \ref{thm:orbitcounting} is a positive integer. This yields the assertion.
\end{proof}

\begin{remark}
In \cite[Algorithm 1 and 2]{NevinsPumpluen2026codes} we present two algorithms to parametrize all the algebras $K[t;\sigma]/K[t;\sigma](t^\de-a)$ up to equivalence and isometry when $\sigma$ is fixed and $K$ finite.  Another parametrization was obtained in \cite{NevinsPumpluen2025}, when $n=m$ are prime, but where only $F$-algebra isomorphisms are considered.
 The two conditions that $m=n$ and that $m$ is prime are then dropped in the parametrizations presented in \cite{NevinsPumpluen2026unpublished.b}. 
 Since Sandler's semifields are not associative and division rings, and the opposite rings of these Petit rings, these parametrizations include all the Petit division rings which are the opposite rings of Sandler's semifields.
 
 By Theorem \ref{thm:allcyclicmain}, for $n=m\geq 2$, every parametrization of the semifields $(K/F,\sigma,a)$
 is also a parametrization of the isotopy classes  of all the semifields
$K[t;\sigma]/K[t;\sigma]f$, where $f$ is irreducible of degree $n$, for fixed $\sigma$, since each such class contains a proper nonassociative cyclic algebra $(K/F,\sigma,a)$.
 \end{remark}

\section{Explicit parametrizations of some nonassociative cyclic algebras in odd characteristic}\label{sec:parametrization}

When $m=n$ is prime we now present a parametrization of some Sandler semifields (or, more precisely, their dual semifields) up to $F$-isomorphisms, which is finer than up to general ring isomorphisms between these rings, but may be more accessible to compute (as there are fewer maps). Note that $F$-isomorphisms are the same as $F$-isometries when $n=m$, that is when  $K[t;\sigma]/K[t;\sigma](t^m-a)=(K/F,\sigma,a)$. A more general parametrization for any $m$, $n$ can be found in \cite{NevinsPumpluen2026unpublished.b}.

Two proper semifields coordinatize the same non-Desarguesian plane if and only if they are isotopic. The subsequent results thus yield a full parametrization of the non-Desarguesian planes related to the semifields $(K/F,\sigma,a)$ of odd prime degree $m$ under $F$-isotopies.

\subsection{When $m$ is odd}

 Let $\mu_m$ be the group of $m$th roots of unity in a fixed algebraic closure of $F$, when the characteristic of $F$ does not divide $m$.

 We will subsequently apply \cite[Theorem 5.2]{NevinsPumpluen2025}  to $F=\mathbb{F}_{q}$ to parametrize the nonassociative cyclic $F$-algebras $(\mathbb{F}_{q^m}/\mathbb{F}_{q},\sigma,a)$ of odd prime degree $m$ up to $F$-isomorphisms. In order to do so, we first collect some terminology from \cite[Theorem 5.2]{NevinsPumpluen2025} and
  note that for finite fields, the norm map is surjective, hence $\normKFim=F^\times$ and so $F^\times/\normKFim=\{1\}$ is trivial.

   Suppose that
 $F=\mathbb{F}_q$ has odd characteristic and contains a primitive $m$th root of unity $\zeta\in F$, which implies that $gcd(m,p)=1$.

  We fix a choice of generator $\sigma$ of ${\rm Gal}(\mathbb{F}_{q^m}/\mathbb{F}_{q})$ for $K=\mathbb{F}_{q^m}$.  Since $\mathbb{F}_{q^m}$ is the splitting field of some polynomial $x^m-b$, with $b\notin (\mathbb{F}_{q}^\times)^m$,  we may choose $\beta\in \mathbb{F}_{q^m}$ to be a root of $x^m-b$ that satisfies $\sigma(\beta) = \zeta \beta$.
Then $$
\{1, \beta, \beta^2, \ldots,\beta^{m-1}\}
$$
is an $\mathbb{F}_{q}$-basis for $\mathbb{F}_{q^m}$ and $\sigma(\beta^k)=\zeta^k\beta^k$ for all $0\leq k <m$.

Let $\mathcal{I} = \mathcal{P}(\{0, 1, \ldots, m-1\}) \smallsetminus \{\{0\},\emptyset\}$ be the set of nonempty subsets of $\{0,1, \ldots,m-1\}$, excluding the set $\{0\}$.  For each $I \in \mathcal{I}$, define
$$K(I) = \{ \sum_{i=0}^{m-1} a_i \beta^i \mid \text{ for all } i\in I, a_i\in \mathbb{F}_{q}^\times, \text{and  for all } i\notin I, a_i=0\} \subset \mathbb{F}_{q^m}.$$ 
Then
    $$
K \smallsetminus F = \bigcup_{I\in \mathcal{I}} K(I).
$$
Let $I\in \mathcal{I}$ be such that $\vert I \vert =k+1\geq 2$. Let
 $\mathbb{F}_{q}^{[\times k]}=\mathbb{F}_{q}^\times\times \mathbb{F}_{q}^\times \times \cdots \times \mathbb{F}_{q}^\times$ be the $k$-fold direct product.
  If the elements of $I$ are $i_0<i_1<\cdots <i_k$ then let
$$
\Delta_I = \{ (\zeta^{s(i_1-i_0)}, \zeta^{s(i_2-i_0)}, \ldots, \zeta^{s(i_k-i_0)}) \in \mathbb{F}_{q}^{[\times k]} \,|\, 1\leq s\leq m\}.
$$
  Write $\mathbb{F}_{q}^{[\times k]}/\Delta_I$ for any fixed choice of representatives for these cosets. For any fixed choice of $a_{i_0}\in F^\times$, define
$$
K(I;a_{i_0}) = \{ a_{i_0}\beta^{i_0} + \sum_{j=1}^{k}a_{i_j}\beta^{i_j} : (a_{i_1},a_{i_2}, \ldots, a_{i_k})\in \mathbb{F}_{q}^{[\times k]}/\Delta_I \}.
$$
Finally, set $K(I;a_{i_0})=\{a_{i_0}\beta^{i_0}\}$ when $I=\{i_0\}\in \mathcal{I}$ and define
$$
\mathcal{S}(\mathbb{F}_{q^m})=\bigcup_{ I\in \mathcal{I} } K(I;1) \subset  \mathbb{F}_{q^m}\smallsetminus \mathbb{F}_{q}.
$$

\begin{theorem} \label{Main:general}
Suppose $m$ is an odd prime, $\mathbb{F}_q$ has odd characteristic and contains a  primitive $m$th root of unity $\zeta$.  Then the distinct $F$-isomorphism classes of nonassociative cyclic algebras of degree $m$  are represented by
$$
(\mathbb{F}_{q^m}/\mathbb{F}_{q},\sigma,a)
$$
where
\\ (i) $ \mathbb{F}_{q^m}$ is the \emph{unique} cyclic Galois field extensions of $\mathbb{F}_{q}$ of degree $m$;
\\ (ii) $\sigma$ is a generator of $\Gal(\mathbb{F}_{q^m}/\mathbb{F}_{q})$;
\\ (iii) $a \in \mathcal{S}(\mathbb{F}_{q^m})$.
\end{theorem}

\subsection{When $m=2$}
Suppose that  $m=2$ and that $\mathbb{F}_q$ has odd characteristic. We obtain from \cite[Theorem 4.2, Corollary 4.3]{NevinsPumpluen2025}:

\begin{theorem}\label{T:p2}
Suppose $\mathbb{F}_{q^2}=\mathbb{F}_q(\sqrt{c})$ for some $c\in \mathbb{F}_q^\times\smallsetminus (\mathbb{F}_q^\times)^2$ and let $\sigma\in \Gal(\mathbb{F}_{q^2}/\mathbb{F}_q)$ be nontrivial.   Then the distinct isomorphism classes of the  nonassociative quaternion algebras with  nucleus $\mathbb{F}_{q^2}$ are represented by $(\mathbb{F}_{q^2}/\mathbb{F}_{q},\sigma,a)$, where $a$ is chosen from the set
$$
\mathcal{S}_2(\mathbb{F}_{q^2})=\{\sqrt{c},  1+s\sqrt{c} \mid  s\in \mathbb{F}_{q}^\times/\{\pm1\}\}.
$$
Alternatively, we can parametrize these distinct isomorphism classes by elements of the set
$$
\mathcal{S}'(\mathbb{F}_{q^2}) =
\{t+\sqrt{c} \mid  t \in \{0\}\cup \mathbb{F}_{q}^\times/\{\pm 1\}\}.
$$
\end{theorem}

\section{Outlook}\label{sec:codes}

\subsection{Possible generalizations to general Petit algebras}
Our results raise several questions. It would certainly be interesting to see in how far our results from Section \ref{sec:main} can be generalized for general Petit algebras $K[t;\sigma]/K[t;\sigma]f$. Since this question touches on our knowledge of the isomorphisms between these algebras, we will likely only be successful when we assume that the degree $m$ of $f$ satisfies  $n=m$ or $n\nmid m$  with $n$ the order of the automorphism $\sigma$ (the order may also be infinite,  $K/F$ does not need to be a Galois field extension).

 The following partial result which is a straightforward generalization of  Theorem  \ref{thm:firstmain} and Theorem \ref{thm:main1} already shows this may be possible in some cases, but it seems very difficult in general.

Let  $K/F$ be a cyclic Galois extension of degree $n$.
 Suppose $\sigma_1, \sigma_2$ are two distinct generators of $\Gal(K/F)$, and let $F_0\subset F$ be any subfield such that $\Aut_{F_0}(K)$ is abelian.

 The proofs that show that  there is no isomorphism of $F_0$-algebras
 $$
    K[t_1;\sigma_1]/K[t_1;\sigma_1](t_1^\de-a_1) \to K[t_2;\sigma_2]/K[t_2;\sigma_2](t_2^\de-a_2)
    $$
 under the assumptions in  Theorem  \ref{thm:firstmain} and Theorem \ref{thm:main1}
 in fact does not use that the algebra $K[t_1;\sigma_1]/K[t_1;\sigma_1](t_1^\de-a_1)$ on the left hand side is defined employing the skew polynomial $f(t)=(t_1^\de-a_1)$, it only uses that $f(t)$ is monic of degree $m$ and not two-sided. This guarantees that the left nucleus of the algebra is $K$, which is used in both proofs.

 Moreover, the proof of Theorem  \ref{thm:firstmain} uses only that
$t_1k=\sigma_1(k)t_1$, and that $t_1^sk=\sigma_1^s(k)t_1^s$ with $t_1^s$
unambiguous for $s\leq m-1$. All of these hold for every  monic $f(t)\in K[t;\sigma_1]$ of degree $m$.

The proof of Theorem \ref{thm:main1} additionally uses
 that we have $t_1t_1^s=t_1^st_1$; again this is also true in any Petit algebra $K[t_1;\sigma_1]/K[t_1;\sigma_1]f(t)$ for every choice of a monic $f(t)\in K[t;\sigma_1]$ of degree $m$.

 This means we have indeed a stronger result.

\begin{theorem} \label{thm:main2}
Let $f(t) \in K[t;\sigma_1] $
be monic of degree $m$ and not right-invariant.  Let
$b\in K^\times$, and let $\Aut_{F_0}(K)$ be abelian. If $n\nmid m$, or if $n=m$ and $b\notin F$, then
$K[t;\sigma_1]/K[t;\sigma_1]f\not\cong K[t;\sigma_2]/K[t;\sigma_2](t^m-b)$ as $F_0$-algebras whenever
$\sigma_2\neq\sigma_1$. 
\end{theorem}

In particular, this implies that when $n\nmid m$, $K[t;\sigma_2]/K[t;\sigma_2](t^m-b)$ is not isomorphic
to \emph{any}  Petit algebra built from a different generator $\sigma_1$ and \emph{any} skew polynomial $f\in K[t;\sigma_1]$ of degree $m$.

We do not expect to obtain similar results for larger classes of Petit algebras $K[t;\sigma_1]/K[t;\sigma_1]f$, since we already know that there are isomorphic algebras that are not isotopic.

\subsection{Skew polycyclic codes}

While we previously focused on irreducible skew polynomials $f\in R=K[t;\sigma,\delta]$,
reducible monic $f\in R=K[t;\sigma,\delta]$ of degree $m$ can be used to build linear skew $(f,\sigma,\delta)$-polycyclic  codes  of length $m$. For a comprehensive ``historical'' exposition, cf. \cite{BL13}.
Skew $(f,\sigma,\delta)$-polycyclic codes are in one-one correspondence with those principal left ideals
of their (potentially) nonassociative Petit ambient ring $R/Rf$, which are generated by a monic right divisor $g$ of $f$ \cite{Pum2025}. In particular,
the cyclic submodules studied in \cite{BU09, BU14, BU14.2} are exactly the principal left ideals in the Petit algebra $R/Rf$,  generated by such a monic right divisor $g$ of $f$, as are
the $(\sigma,\delta)$-codes of \cite{BL13}. When $f(t)=t^m-a \in R=K[t;\sigma]$ as in our setup, then $\mathcal{C}$ is a skew $(\sigma,a)$-constacyclic code  of length $m$.

Previous attempts at classifications of skew polynomial codes always only considered some fixed $\sigma\in \Gal(K/F)$.  The main idea is that  $(\sigma,a)$-constacyclic codes and $(\sigma,b)$-constacyclic codes belonging to the same equivalence class have the same minimum distance and the same algebraic structures.
First steps to classify skew $(\sigma,a)$-constacyclic codes by some natural equivalence relation that preserves their Hamming distance have  been achieved in \cite{Ou2025}. The approach defines equivalences using a subclass of isomorphisms between the ambient Petit rings $\mathbb{F}_{q^n}[t;\sigma]/\mathbb{F}_{q^n}[t;\sigma](t^m-a)$ and $\mathbb{F}_{q^n}[t;\sigma]/\mathbb{F}_{q^n}[t;\sigma](t^m-b)$ that preserve the Hamming distance, and restrict to the identity map on $K=\mathbb{F}_{q^n}$.  The approach has recently been extended in \cite{Pum2025} and now includes equivalences defined by using  isomorphisms between the ambient Petit rings which restrict to some $\tau\in \Aut(K)$, and that preserve the Hamming distance. Equivalent codes, however, are being built with the same fixed $\sigma$ that generates the Galois group $\Gal(\mathbb{F}_{q^n}/\mathbb{F}_{q})$.

Now we have gained some insight into what we can and cannot say when we have two different $\sigma_1,\sigma_2\in \Gal(K/F)$, and want to compare skew $(\sigma_1,a_1)$-constacyclic codes of length $m$ with skew $(\sigma_2,a_2)$-constacyclic codes of length $m$. We can define code equivalences by putting two skew polynomial codes into the same equivalence class, if they correspond to two isomorphic principal left ideals, and the isomorphism preserves the Hamming distance.

Let $R_i=\mathbb{F}_{q^{n}}[t;\sigma_i]$ and $i\in\{1,2\}$. We observe that for two ambient rings to be isomorphic they must have the same center, which immediately forces the condition that $\Gal(\mathbb{F}_{q^n}/\mathbb{F}_{q}) =\langle\sigma_1 \rangle=\langle\sigma_2 \rangle$ and so both $\sigma_1$ and $\sigma_2$ have the same order. If this condition is not satisfied, we cannot compare skew constacyclic codes that have these two different ambient rings. So we assume $\Gal(\mathbb{F}_{q^n}/\mathbb{F}_{q})=\langle\sigma_1 \rangle=\langle\sigma_2 \rangle$ and $\sigma_i$ has order $n$.

Assume first that $a_i\in \mathbb{F}_{q}^\times$ and that $n\mid m$ for the order $n$ of $\sigma_i$. Then the Petit rings $\mathbb{F}_{q^n}[t;\sigma_i]/\mathbb{F}_{q^n}[t;\sigma_i](t^m-a_i)$ are associative ambient rings of skew $(\sigma_i,a_i)$-constacyclic codes of length $m$.

Moreover, we have
$$\mathbb{F}_{q^n}[t;\sigma_i]/\mathbb{F}_{q^n}[t;\sigma_i](t^m-a_i)\cong \mathbb{F}_{q^n}[t;\sigma_i^\ell]/\mathbb{F}_{q^n}[t;\sigma_i^\ell](t^m-a_i^\ell)$$
for all integers $\ell<n$ where $\sigma_i^\ell$ also generates the Galois group, that means $\gcd(\ell,n)=1$.

In this case we can use the Hamming-weight preserving isomorphisms $H$ between the ambient rings $\mathbb{F}_{q^n}[t;\sigma^\ell]/\mathbb{F}_{q^n}[t;\sigma_i^\ell](t^m-a_i^\ell)$ and  $\mathbb{F}_{q^n}[t;\sigma_i]/\mathbb{F}_{q^n}[t;\sigma_i](t^m-a_i)$, e.g. $H|_K=id$, $H(t)=\alpha t^\ell$ with $\gcd(\ell,m)=1$, to compare and classify
isometric or equivalent skew $(\sigma_i,a_i)$-constacyclic codes and skew $(\sigma_i^\ell,a_i^\ell)$-constacyclic codes of length $m$. Indeed,
we can classify skew $(\sigma_1,a_1)$-constacyclic codes and skew $(\sigma_2,a_2)$-constacyclic codes of length $m$ by equivalence classes, whenever we have associative ambient rings: suppose that $\Gal(\mathbb{F}_{q^n}/\mathbb{F}_{q})=\langle\sigma_i\rangle$, and that
$$\mathbb{F}_{q^n}[t;\sigma_1]/\mathbb{F}_{q^n}[t;\sigma_1](t^m-a_1)\cong \mathbb{F}_{q^n}[t;\sigma_2]/\mathbb{F}_{q^n}[t;\sigma_2](t^m-a_2).$$
Then there exists an integer $\ell$ such that $\sigma_2=\sigma_1^\ell$ and so
$$\mathbb{F}_{q^n}[t;\sigma_1]/\mathbb{F}_{q^n}[t;\sigma_1](t^m-a_1)\cong \mathbb{F}_{q^n}[t;\sigma_1^\ell]/\mathbb{F}_{q^n}[t;\sigma_1^\ell](t^m-a_2)$$
and we also know that
$$ \mathbb{F}_{q^n}[t;\sigma_1]/\mathbb{F}_{q^n}[t;\sigma_1](t^m-a_1)\cong \mathbb{F}_{q^n}[t;\sigma_1^\ell]/\mathbb{F}_{q^n}[t;\sigma_1^\ell](t^m-a_1^\ell).$$
This means that for associative ambient algebras, we can use the equivalence classes defined in \cite{Pum2027} and easily obtain results also for the equivalence of skew constacyclic codes that are defined with different generators of the Galois group. All the Hamming weight preserving isomorphisms are well known. For instance, the skew $(\sigma_1,a_1)$-constacyclic codes are isometric/equivalent to the skew $(\sigma_2,a_2)$-constacyclic codes, if and only if the skew $(\sigma_1^\ell,a_1^\ell)$-constacyclic codes are isometric/equivalent to the skew $(\sigma_1^\ell,a_2)$-constacyclic codes.

Assume next that either $a_i\in \mathbb{F}_{q^n}\smallsetminus \mathbb{F}_{q}$ or that $n\nmid m$. Then the Petit ambient rings $\mathbb{F}_{q^n}[t;\sigma_i]/\mathbb{F}_{q^n}[t;\sigma_i](t^m-a_i)$ are proper nonassociative.  Assume that $n\nmid m$ or $n=m$.
Since
$$
\mathbb{F}_{q^n}[t;\sigma_1]/\mathbb{F}_{q^n}[t;\sigma_1](t^m-a_1)\not\cong \mathbb{F}_{q^n}[t;\sigma_2]/\mathbb{F}_{q^n}[t;\sigma_2](t^m-a_2)
$$
for any two distinct generators $\sigma_1$ and $ \sigma_2$ of  $\Gal(\mathbb{F}_{q^n}/\mathbb{F}_{q})$ and any choice of $a_i\in \mathbb{F}_{q^n}\smallsetminus \mathbb{F}_{q}$, the choice of generator $\sigma_i$ will also affect the skew $(\sigma_i,a_i)$-constacyclic codes we obtain. The skew $(\sigma_1,a_1)$-constacyclic codes and the skew $(\sigma_2,a_2)$-constacyclic codes correspond to the  principal left ideals of two non-isomorphic algebras. We cannot easily compare them.

\subsection*{Acknowledgments} The author thanks the Department of Mathematics and Statistics  at the University of Ottawa  for its hospitality and its congenial and inspiring atmosphere, and Monica Nevins for the many stimulating discussions which led to an improvement of the original version of this paper.

\section*{Compliance with ethical standards}
\subsection*{Ethics and consent to participate} N/A
\subsection*{Availability of data and materials}  No data were generated or used for this research.

\subsection*{Funding sources}  Parts of this paper were written during the author's stay as a Simons Professor in Residence at the University of Ottawa in July and August 2025. She gratefully acknowledges the support of CRM and the Simons Foundation.

\subsection*{Disclosure of potential conflicts of interest} The author has no relevant financial or non-financial interests to disclose.

\subsection*{Data Availability Statement} No data were generated or used for this research.

\providecommand{\bysame}{\leavevmode\hbox to3em{\hrulefill}\thinspace}
\providecommand{\MR}{\relax\ifhmode\unskip\space\fi MR }
\providecommand{\MRhref}[2]{%
  \href{http://www.ams.org/mathscinet-getitem?mr=#1}{#2}
}
\providecommand{\href}[2]{#2}

\end{document}